\documentclass[10pt]{amsart}
\usepackage{amssymb}
\usepackage{dsfont}
\usepackage{amscd}
\usepackage[mathscr]{euscript} 
\usepackage[all]{xy}
\usepackage{url}
\usepackage{stmaryrd}
\usepackage{comment}
\usepackage[retainorgcmds]{IEEEtrantools}
\usepackage{hyperref}
\usepackage{graphicx}
\usepackage{mathtools}
\usepackage[dvipsnames]{xcolor}
\usepackage{centernot}
\usepackage{marvosym}
\usepackage{tikz}
\usepackage{tikz-cd}
\usepackage{xspace}
\usepackage{bbm}
\usepackage[title]{appendix}

\usepackage{tkz-fct}
\usepackage{soul}
\usepackage{lineno}
\usepackage[T1]{fontenc}
\usepackage[utf8]{inputenc}

\usepackage[english]{babel}
\usepackage[autostyle]{csquotes}
\usepackage[shortlabels]{enumitem}
\usepackage{todonotes}
\usepackage[realmainfile]{currfile}

\usepackage[backend=biber, style=alphabetic]{biblatex}
\allowdisplaybreaks 

\title[Idempotent and invariant measures]{On idempotent measure conjecture and decomposition of invariant measures}


\author[D. M. HOFFMANN]{Daniel Max Hoffmann$^{\dagger}$}
\thanks{$^{\dagger}$SDG. The first author is supported
by the National Science Centre (Narodowe Centrum Nauki, Poland)
grant no. 2021/43/B/ST1/00405.}
\address{$^{\dagger}$
Instytut Matematyki\\
Uniwersytet Warszawski\\
Warszawa\\
Poland}
\email{daniel.max.hoffmann@gmail.com}
\urladdr{\href{https://orcid.org/0000-0002-4514-269X}{\includegraphics[height=\fontcharht\font`\B]{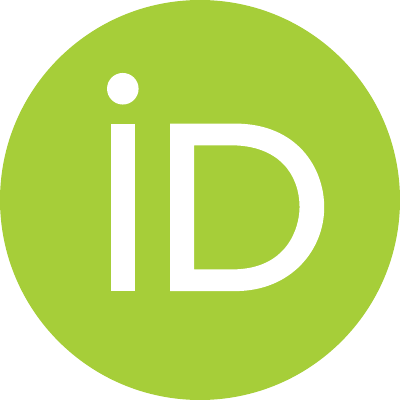} {0000-0002-4514-269X}}\\
\href{https://sites.google.com/site/danielmaxhoffmann/}{https://sites.google.com/site/danielmaxhoffmann/}}

\author[T. Rzepecki]{Tomasz Rzepecki$^{\ddagger}$}
\address{$^{\ddagger}$
Instytut Matematyczny\\
Uniwersytet Wrocławski\\
Wrocław\\
Poland}
\email{tomasz.rzepecki@math.uni.wroc.pl}
\urladdr{\href{https://orcid.org/0000-0001-9786-1648}{\includegraphics[height=\fontcharht\font`\B]{orcidlogo.pdf} {0000-0001-9786-1648}}\\
\href{https://fricas.org/~rzepecki/}{https://fricas.org/~rzepecki/}}

\date{\today}

 \DeclareMathOperator{\aut}{Aut} \DeclareMathOperator{\id}{id}

\DeclareMathOperator{\stab}{Stab}

 \DeclareMathOperator{\gal}{Gal}

\DeclareMathOperator{\df}{def}

\DeclareMathOperator{\tp}{tp}

\DeclareMathOperator{\Borel}{Borel}

\DeclareMathOperator{\res}{res}

\DeclareMathOperator{\Aut}{Aut}

\DeclareMathOperator{\fGen}{fGen}

\DeclareMathOperator{\Rr}{\mathbb{R}}

\newcommand{\inv}{\mathrm{inv}}

\newtheorem{theorem}{Theorem}[section]
\newtheorem{proposition}[theorem]{Proposition}
\newtheorem{lemma}[theorem]{Lemma}
\newtheorem{cor}[theorem]{Corollary}
\newtheorem{fact}[theorem]{Fact}
\theoremstyle{definition}
\newtheorem{definition}[theorem]{Definition}

\newtheorem{remark}[theorem]{Remark}
\newtheorem{question}[theorem]{Question}
\newtheorem{conj}[theorem]{Conjecture}

\theoremstyle{remark}
\newtheorem*{theorem*}{Theorem}
\newtheorem*{cor*}{Corollary}

\theoremstyle{definition}
\theoremstyle{definition}

\theoremstyle{definition}

\theoremstyle{remark}

\newtheorem*{clm*}{Claim}

\AtEndEnvironment{proof}{\setcounter{clm}{0}}

\makeatletter
\providecommand*{\cupdot}{%
  \mathbin{%
    \mathpalette\@cupdot{}%
  }%
}
\newcommand*{\@cupdot}[2]{%
  \ooalign{%
    $\m@th#1\cup$\cr
    \sbox0{$#1\cup$}%
    \dimen@=\ht0 %
    \sbox0{$\m@th#1\cdot$}%
    \advance\dimen@ by -\ht0 %
    \dimen@=.5\dimen@
    \hidewidth\raise\dimen@\box0\hidewidth
  }%
}

\providecommand*{\bigcupdot}{%
  \mathop{%
    \vphantom{\bigcup}%
    \mathpalette\@bigcupdot{}%
  }%
}
\newcommand*{\@bigcupdot}[2]{%
  \ooalign{%
    $\m@th#1\bigcup$\cr
    \sbox0{$#1\bigcup$}%
    \dimen@=\ht0 %
    \advance\dimen@ by -\dp0 %
    \sbox0{\scalebox{2}{$\m@th#1\cdot$}}%
    \advance\dimen@ by -\ht0 %
    \dimen@=.5\dimen@
    \hidewidth\raise\dimen@\box0\hidewidth
  }%
}
\makeatother

\def\Ind#1#2{#1\setbox0=\hbox{$#1x$}\kern\wd0\hbox to 0pt{\hss$#1\mid$\hss}
\lower.9\ht0\hbox to 0pt{\hss$#1\smile$\hss}\kern\wd0}

\def\notind#1#2{#1\setbox0=\hbox{$#1x$}\kern\wd0
\hbox to 0pt{\mathchardef\nn=12854\hss$#1\nn$\kern1.4\wd0\hss}
\hbox to 0pt{\hss$#1\mid$\hss}\lower.9\ht0 \hbox to 0pt{\hss$#1\smile$\hss}\kern\wd0}

\newcommand{\FC}{\mathfrak{C}}

\newcommand{\CG}{{\mathcal G}}

\newcommand{\CL}{{\mathcal L}}

\newcommand{\CF}{{\mathcal F}}

\newcommand{\fs}{\mathrm{fs}}
\newcommand{\sfs}{\mathrm{sfs}}
\newcommand{\conv}{\mathrm{conv}}
\newcommand{\Av}{\mathrm{Av}}
\newcommand{\supp}{\mathrm{supp}}

\newcommand{\autf}{\mathrm{Autf}}
\newcommand{\KP}{\mathrm{KP}}
\newcommand{\Las}{\mathrm{Lasc}}

\begin{document}

\begin{abstract}
We continue the study of the semigroup of global invariant types
introduced by Gannon, Hoffmann, and Krupi\'nski and the associated
convolution semigroup of invariant Keisler measures.

The first part of the paper concerns the Idempotent Measure
Conjecture, studied in \cite{CGK,GHK}, which predicts that
idempotent fim Keisler measures should be precisely the invariant Haar
measures on relatively type-definable subgroups.
We identify structural properties underlying all previously known
positive instances of the conjecture and prove a conditional version
under natural semigroup-theoretic hypotheses. In particular, we show
that left simplicity of the support and a local invariance condition
are sufficient for the conjectural characterization.

The second part studies invariant Keisler measures in amenable NIP
theories. We prove that the semigroup of global invariant types has a
unique minimal left ideal consisting of f-generic types and identify
its Ellis groups. Using normalized Haar measure on
$\gal_{\KP}(T)$, we associate canonically an invariant Keisler measure to
every Ellis group. We show that the resulting measures are supported
on minimal subflows, prove that the flow
$(S_{\bar n}(\mathfrak C),\Aut(\mathfrak C))$ is hereditarily
amenable, establish unique ergodicity of every minimal subflow in the
countable case, and characterize all ergodic
$\Aut(\mathfrak C)$-invariant Keisler measures as the measures arising
from this construction.
\end{abstract}

\maketitle

\section{Introduction}

In \cite{GHK},
the first author, Gannon, 
and Krupiński introduced a canonical
semigroup structure on the space
$$S_{\bar m}^{\inv}(\FC,M)$$
of global $M$-invariant types extending the type of a model, together
with an affine convolution operation on the corresponding space of
Keisler measures. The present paper studies the minimal ideals of this
semigroup and applies their structure in two directions: to the
Idempotent Measure Conjecture and to the ergodic decomposition of invariant
Keisler measures in amenable NIP theories.

\subsection{Overview}

Model-theoretic topological dynamics studies actions of definable
groups and automorphism groups on compact spaces of types and
measures. The subject originated in the work of Newelski
\cite{newelski09}, which revealed substantial connections between
topological dynamics and model theory. Historically, the action of a
definable group $G$ on the space of types concentrated on $G$ was
studied first and in greater detail; see, for example,
\cite{newelski09,newelski12,Anand2013,GPPb,KruPil17,WAPstable}.

The introduction of Keisler measures \cite{Keisler1987} into this
framework proved particularly fruitful; see
\cite{NIP1,Anand_Udi2011,HruPiSi13,HruKruPi2}. In the definable-group
setting, Chernikov and Simon \cite{ArtemPierre} described the ergodic
measures for definably amenable NIP groups and obtained a positive
solution to the original Newelski conjecture. Another related line of
work concerns definable convolution of Keisler measures
\cite{Artem_Kyle,Artem_Kyle2,CGK}. The classification of
convolution-idempotent measures is motivated by classical results
establishing a correspondence between compact subgroups and
idempotent probability measures; see
\cite{kawada1940probability,Wendel,Rudin,Glicksberg2,Cohen,Pym62}.

Actions of automorphism groups have received increasing attention in
recent years; see
\cite{Krupinski2018,Krupinski2019,KrRz,HruKruPi}. The actions of
closed subgroups of an automorphism group, in the relatively definable
topology, includes
the actions of definable groups
through the affine sort
construction introduced in \cite{Gismatullin2008} and developed
further in \cite[Section 5]{GHK}. Many results for definable groups
therefore have natural automorphism-dynamical counterparts, although
their proofs often require additional techniques. One example is the
automorphism version of Newelski's group chunk theorem
\cite[Theorem 7.25]{GHK}. Another is the connection between
Hrushovski's definability patterns \cite{Hrushovski2019} and
automorphism dynamics, which led to the construction of the
$\ast$-product in \cite{GHK}. The resulting semigroup of invariant
types appears to encode significant information about the theory
$T$; the results below provide further evidence for this point of
view.

\subsection{Definable groups versus automorphism groups}
We first recall the setting with a definable group. Let $G$ be a group
definable in an $\CL$-theory $T$, and let $M\models T$ contain the
parameters occurring in the definition of $G$. Suppose that every
type in $S_G(M)$ has a unique global coheir extension. This holds, for
example, for $M=\mathbb R$ as a field and, more generally, after
passing to the Shelah expansion in the NIP setting.

Under these assumptions, the natural action of $G(M)$ on $S_G(M)$
extends to a left-continuous associative operation $\ast$, the
Newelski product. More generally, without uniqueness of coheirs, the
same construction yields a left-continuous semigroup operation on
the compact space
$$S_G^{\fs}(\FC,M)$$
of global types finitely satisfiable in $M$. The Ellis semigroup of
the $G(M)$-ambit $S_G(M)$ is naturally isomorphic to this semigroup:
$$\big(E(S_G(M),G(M)),\circ\big)\cong\big(S_G^{\fs}(\FC,M),\ast\big);$$
see \cite[Fact 1.6]{Anand2013}, as well as
\cite[Fact 6.1]{newelski09} and
\cite[Theorem 6.10]{Artem_Kyle}.

The product also admits the explicit description
$$p\ast q=\tp(a\cdot b/\FC),$$
where $b\models q$ and $a\models p|_{\FC b}$. This formula extends the
operation from finitely satisfiable types to the larger compact space
$S_G^{\inv}(\FC,M)$ of global $M$-invariant types concentrating on
$G$. Its affine analogue is the definable convolution of
$M$-invariant Keisler measures, which gives the compact left
semigroup $\mathfrak M_G^{\inv}(\FC,M)$ and its subsemigroup
$\mathfrak M_G^{\fs}(\FC,M)$; see
\cite{Artem_Kyle,Artem_Kyle2}.

The $\ast$-product leads naturally to the problem of classifying its
idempotents. In the classical setting, a regular Borel probability
measure on a locally compact group is convolution-idempotent precisely
when it is the normalized Haar measure of a compact subgroup; see
\cite[Theorem A.4.1]{Pym62}. Motivated by this analogy, Chernikov,
Gannon, and Krupiński formulated the following conjecture in
\cite{CGK}. Here fim measures, defined in
Definition~\ref{def:fim}, are measure-theoretic counterparts of
generically stable types.

\begin{conj}\label{conjecture: main conjecture for definable groups}
Let $G=G(\FC)$ be a definable group and let
$\mu\in\mathfrak M_G^{\df}(\FC,M)$
(i.e.\ an $M$-definable measure concentrated on $G$)
be fim over $M$. The right
stabilizer $\stab(\mu)$ is then type-definable over $M$. The following
conditions are conjectured to be equivalent:
\begin{enumerate}
    \item
    $\mu$ is idempotent with respect to definable convolution;

    \item
    $\mu$ is the unique right $G$-invariant, equivalently the unique
    left $G$-invariant, Keisler measure concentrated on
    $\stab(\mu)$.
\end{enumerate}
In particular, the conjecture predicts a correspondence between
idempotent fim measures in
$\mathfrak M_G^{\inv}(\FC,M)$ and $M$-type-definable fim subgroups
of $G(\FC)$; cf.\ \cite[Definition 2.24]{GHK}.
\end{conj}

The conjecture has been confirmed in several cases: for groups
definable in stable theories, for abelian groups, for groups definable
in NIP theories under an additional $G^{00}$-invariance assumption,
and for types in rosy theories.

We now turn to automorphism groups. Let $\FC\models T$ be a monster
model, let $M\preceq\FC$ be small, and let $\bar m$ enumerate $M$.
By \cite[Theorems 4.9 and 4.14]{GHK}, in NIP one has
$$\big(E(S_{\bar x}(M),\aut(M)),\circ\big)\cong\big(S_{\bar m}^{\sfs}(\FC,M),\ast\big),$$
where $S_{\bar m}^{\sfs}(\FC,M)$ is the space of global types
extending $\tp(\bar m/\emptyset)$ that are strongly finitely
satisfiable in $M$; see \cite[Definition 2.16]{GHK}\footnote{We use strong finite
satisfiability only in the displayed semigroup identification.}.
The same definition yields a compact Hausdorff
left-continuous semigroup
$$\big(S_{\bar m}^{\inv}(\FC,M),\ast\big),$$
and an affine $\ast$-product on Borel $M$-definable Keisler measures
concentrated on $\tp(\bar m/\emptyset)$.

The corresponding automorphism version of the conjecture is the
following.

\begin{conj}[Idempotent Measure Conjecture]
\label{conjecture: main conjecture in the introduction}
Let $\mu\in\mathfrak M_{\bar m}^{\df}(\FC,M)$ be fim over $M$. By
\cite[Lemma 2.26]{GHK},
$$\stab_l(\mu)=G_{\pi,\FC}$$
for some partial type
$\pi(\bar x;\bar y)\vdash\bar x\equiv_{\emptyset}\bar y$.
The following conditions are conjectured to be equivalent:
\begin{enumerate}
    \item
    $\mu$ is idempotent;

    \item
    $\mu$ is the unique left $G_{\pi,\FC}$-invariant measure in
    $\mathfrak M_{\pi(\bar x;\bar m)}^{\inv}(\FC,M)$.
\end{enumerate}
In particular, the conjecture predicts a correspondence between
idempotent fim measures in
$\mathfrak M_{\bar m}^{\inv}(\FC,M)$ and relatively
$\bar m$-type-definable over $M$ fim subgroups of $\aut(\FC)$.
\end{conj}

This conjecture sheds light on relatively type-definable subgroups of
$\aut(\FC)$ and is therefore relevant to the study of automorphism
groups. It was confirmed in \cite{GHK} for stable theories, for NIP
theories under the additional assumption of KP-invariance, and for
types in rosy theories.

\subsection{Results of this paper}
We obtain results in two related directions.

\subsubsection*{The Idempotent Measure Conjecture}
Our analysis of \cite{Pym62,CGK,GHK} isolates dependencies among the
hypotheses used in the known positive cases of the two conjectures
above. This leads to Theorem~\ref{thm: weak conjecture}, a version of
the Idempotent Measure Conjecture under additional hypotheses. We do
not claim essentially new proofs of the known cases; rather, our aim
is to identify their common structural ingredients and in this way 
give more concrete targets for future work on the full conjecture.

The first ingredient is left-simplicity of the support. We show that
condition (2) in the Idempotent Measure Conjecture can hold only when
$\supp(\mu)$ is a left-simple semigroup; see
Corollary~\ref{cor:supp_cont_pi_ls} and the surrounding discussion.
In the definable-group terminology of \cite{CGK}, left-simplicity is
equivalent to support transitivity for idempotent fim measures
\cite[Proposition 3.51]{CGK}. It holds in all currently known positive
cases; see, for example, \cite[Proposition 3.55]{CGK} and
\cite[Theorem 6.31 and Corollary 8.25]{GHK}.

The second ingredient is the local invariance property
\eqref{eq:restricted_lascar_inv}, extracted from the stable proofs.
This property concerns invariance over small models and holds, for
example, for Lascar-invariant measures; see
Remark~\ref{rem: diamond and Lascar}. Lemma~\ref{lemma: diamond}
shows that condition (2) of the conjecture implies
\eqref{eq:restricted_lascar_inv}. Conversely, for an idempotent
superfim measure with left-simple support, condition
\eqref{eq:restricted_lascar_inv} implies
the required support containment
$\supp(\mu)\subseteq[\pi(\bar{x};\bar{m})]$, where
$G_{\pi,\FC}=\stab_l(\mu)$.
Combining these observations gives
Theorem~\ref{thm: weak conjecture} and its type analogue,
Theorem~\ref{thm: weak conjecture_types}. The known stable and
KP-invariant cases may be viewed as instances in which the additional
hypotheses are already verified. This leads naturally to
Questions~\ref{question: diamond},~\ref{question: diamond2},
and~\ref{question: left-simple}.

\subsubsection*{Decomposition of invariant measures}
The second part of the paper concerns $\aut(\FC)$-invariant Keisler
measures in amenable NIP theories. Such measures form an important
class of idempotent measures.

We first introduce f-generic types for the automorphism action, using
the results of \cite{HruKruPi}. In the NIP setting, these are precisely
the global types that do not fork over $\emptyset$. In
Theorem~\ref{thm: fGen and injectivity on Ellis groups}, we prove that
the f-generic types form the unique minimal left ideal of
$$\big(S_{\bar n}^{\inv}(\FC,M),\ast\big)$$
and that their $\aut(\FC)$-orbits are exactly the Ellis groups. The
argument parallels the definable-group approach of Gannon and
Rzepecki \cite[Section 3]{GR25}. We also describe, in
Corollary~\ref{cor: description of idempotent measures with f-generics},
the semigroup structure of the support of an idempotent measure whose
support contains an f-generic type.

A central role is played by the map
$$\rho\colon S_{\bar n}^{\inv}(\FC,M)\longrightarrow\gal_{\KP}(T),
\qquad\rho\big(\tp(\sigma(\bar n)/\FC)\big)=
\sigma^{-1}\autf_{\KP}(\FC').$$
The map $\rho$ is injective on every Ellis group. For each f-generic
type $p$, we use the normalized Haar measure on $\gal_{\KP}(T)$ to define
a canonical invariant Keisler measure $\mu_p$ associated with the
Ellis group $p\ast S$; see Definition~\ref{def: almost pullback}.

These measures yield several dynamical consequences. First,
Theorem~\ref{thm: hereditary amenability of type flow} shows that the
full flow
$$\big(S_{\bar n}(\FC),\aut(\FC)\big)$$
is hereditarily amenable: every minimal subflow consists of
f-generic types and is the support of
an invariant measure
measure $\mu_p$\footnote{No countability assumption is needed for this conclusion.}.

When the language is countable, we obtain stronger results.
Corollary~\ref{cor: minimal subflows uniquely ergodic} shows that every
minimal subflow of $S_{\bar n}(\FC)$ is uniquely ergodic
(i.e.\ the measure $\mu_p$ mentioned above is the unique invariant measure in that subflow).
Finally,
Theorem~\ref{thm: ergodic description} identifies all ergodic
$\aut(\FC)$-invariant Keisler measures as the measures $\mu_p$.
Lemma~\ref{lemma: approximation 2} shows more generally that every
invariant measure can be approximated by finite convex combinations
of these canonical components. This is the automorphism-dynamical
analogue of the definable-group decomposition obtained in
\cite{ArtemPierre}.

The affine sort construction suggests a further extension to actions
of relatively type-definable closed subgroups of $\aut(\FC)$. We do
not pursue this generalization here in order to keep the exposition
focused. The other principal instance, namely definable-group
actions, is already treated in \cite{ArtemPierre} without assuming
that the language is countable.

\subsection{Structure}
Section~2 introduces the $\ast$-product and records the basic facts
about invariant types and Keisler measures used later. Section~3
studies supports and stabilizers of idempotent measures and proves the
conditional version of the Idempotent Measure Conjecture.
Section~4 develops the theory of f-generic types and identifies the
minimal ideal of the invariant-type semigroup. Section~5 constructs
the measures $\mu_p$, proves hereditary amenability and unique
ergodicity of minimal subflows, and classifies the ergodic invariant
measures. The appendix records the semigroup-theoretic facts used in
the main text.

\subsection{Acknowledgements}
We would like to thank Krzysztof Krupiński for several helpful
discussions and for finding a mistake in an early draft of the paper,
and Kyle Gannon for sharing his insights with the second author.

\section{Preliminaries}
Consider an $\CL$-theory $T$ with monster model $\FC$ and a bigger monster model $\FC'\succeq\FC$.
Fix a small model $M\preceq\FC$ and its enumeration $\bar{m}$, an enumeration $\bar{c}\supseteq\bar{m}$ of $\FC$ and let $\bar n$ with $\bar{m}\subseteq\bar{n}\subseteq\bar{c}$ enumerate a model $N$ with $M\preceq N\preceq \FC$ (possibly $M=N$ or $N=\FC$).
Consider tuples of variables $\bar{x}$ and $\bar{y}$ such that $\bar{n}\in\FC^{\bar{x}}$ and $\bar{n}\in\FC^{\bar{y}}$.
Furthermore, let us fix tuples of variables $\bar{x}'\subseteq\bar{x}$ and $\bar{y}'\subseteq\bar{y}$
which correspond to the tuple $\bar{m}$ (note that in the case of $\bar{n}=\bar{m}$ we have that $\bar{x}'=\bar{x}$ and $\bar{y}'=\bar{y}$).
Because $\bar{y}$ is a tuple also corresponding to an enumeration of model $N$, every formula $\psi(\bar{x};\bar{b})$, with $\bar{b}$ from $\FC$, can be rearranged so that $\bar{b}\in \FC^{\bar{y}}$.
We will often use this without further explanation.

We consider types extending the type $\tp(\bar{n}/\emptyset)$, i.e.\ a closed subset of the space $S_{\bar{x}}(\FC)$ denoted by ``$S_{\bar{n}}(\FC)$''. In a similar way, we consider the space $\mathfrak{M}_{\bar{n}}(\FC)$ of Keisler measures concentrated on the type $\tp(\bar{n}/\emptyset)$.
More generally, if $F(\bar{x})$ is a partial type (with parameters from $\FC$) then $S_F(\FC)$ stands for the closed subset of $S_{\bar{x}}(\FC)$ of global types extending the partial type $F(\bar{x})$.
There are other spaces involving this notational rule, e.g.\ $\mathfrak{M}^{\inv}_{F}(\FC,M)$ stands for the space of $M$-invariant Keisler measures concentrated on the closed subset $S_F(\FC)$.
We hope it is clear and does not need further explanation.

We need a slight generalization of the definition of the $\ast$-product of Keisler measures from \cite{GHK},
where we allow measures and types in longer tuples of variables (e.g.\ over tuples enumerating $\FC$).
Let us introduce the needed notions now.
The definition involves a map given as follows.
For each type $p(\bar{x})\in S_{\bar{n}}(\FC)$ there exists $\sigma\in\aut(\FC')$ such that $p(\bar{x})=\tp(\sigma(\bar{n})/\FC)$.
Take $\bar{b}\in \FC^{\bar{y}}$ and consider the following map
$$h_{\bar{b}}\colon S_{\bar{n}}(\FC)\to S_{\bar{y}}(N)$$
defined by:
$$h_{\bar{b}}:\tp(\sigma(\bar{n})/\FC)\mapsto\tp(\sigma^{-1}(\bar{b})/N).$$
It turns out that $h_{\bar{b}}$ is a well-defined continuous map. For example, continuity follows by
$$h_{\bar{b}}^{-1}[\theta(\bar{n};\bar{y})]=[\theta(\bar{x};\bar{b})]\subseteq S_{\bar{n}}(\FC),$$
where $\theta(\bar{x};\bar{y})\in\CL$.

\begin{remark}
    Let $\mu\in\mathfrak{M}_{\bar{n}}(\FC)$, $\theta(\bar{x};\bar{y})\in\CL$ and $\bar{b}\in\FC^{\bar{y}}$. Then
    $$\big((h_{\bar{b}})_{\ast}\mu\big)\big(\theta(\bar{n};\bar{y})\big)=\mu\big(\theta(\bar{x};\bar{b})\big).$$
\end{remark}

We recover the original setting of \cite{GHK} if $N=M$.
However, there is another useful case when $N=\FC$ (and then $\bar{n}=\bar{c}$) for which we define the inverse of a Keisler measure:

\begin{definition}\label{def:inverse}
    For $\mu\in\mathfrak{M}_{\bar{c}}(\FC)$, we set
    $$\mu^{-1}:=(h_{\bar{c}})_{\ast}\mu\in\mathfrak{M}_{\bar{y}}(\FC).$$
\end{definition}
\noindent
In fact, in the above definition, we have $\mu^{-1}\in \mathfrak{M}_{\bar{c}}(\FC)$.

Now, let $\res_M\colon S_{\bar{x}}(N)\to S_{\bar{x}}(M)$ denote the standard restriction map.
Sometimes, to keep track of the variables appearing in Keisler measures, we use a subscript:
``$\mu_{\bar{y}}$'' means that the measure $\mu$ is considered in variables $\bar{y}$.
We consider the Morley product $\otimes$ of Keisler measures (and the Morley product of invariant types) and so-called \emph{fiber functions} $F_{\mu}^{\varphi}$. These, and many more, are standard notions in the topic of Keisler measures. The reader may consult \cite{Guide_NIP} or \cite{KyleThesis} for a nice exposition on Keisler measures.

\begin{definition}\label{def:star.definition}
    Let
    $\mu\in\mathfrak{M}_{\bar{x}}^{\inv}(\FC,M)$
    be Borel $M$-definable
    and let
    $\nu\in\mathfrak{M}_{\bar{n}}(\FC)$ (in variable $\bar{x}$).
We define a measure $\mu\ast\nu\in\mathfrak{M}_{\bar{x}}(\FC)$  as follows:
\begin{IEEEeqnarray*}{rCl}
(\mu\ast\nu)\big(\theta(\bar{x};\bar{b})\big) &:=&
\Big(\mu_{\bar{x}}\otimes \big( (h_{\bar{b}})_\ast(\nu_{\bar{x}}) \big)_{\bar{y}} \Big)\big( \theta(\bar{x};\bar{y})\big) \\
&=& \int\limits_{q(\bar{y})\in S_{\bar{y}}(M)} F^{\theta(\bar{x};\bar{y})}_{\mu_{\bar{x}}}(q)\,d(\res_M)_\ast\Big(\big( (h_{\bar{b}})_\ast(\nu_{\bar{x}})\big)_{\bar{y}}\Big) \\
&=& \int\limits_{q(\bar{x})\in S_{\bar{x}}(\FC)} \Big( F^{\theta(\bar{x};\bar{y})}_{\mu_{\bar{x}}}\circ\res_M\circ h_{\bar{b}}\Big)(q)\,d\nu_{\bar{x}} \\
&=& \int\limits_{q(\bar{x})\in S_{\bar{n}}(\FC)} \Big( F^{\theta(\bar{x};\bar{y})}_{\mu_{\bar{x}}}\circ\res_M\circ h_{\bar{b}}\Big)(q)\,d\nu_{\bar{x}}
\end{IEEEeqnarray*}
where $\bar{b}\in \FC^{\bar{y}}$ and $\theta(\bar{x};\bar{y})\in\CL_{\bar{x};\bar{y}}$.
\end{definition}

Many properties of Keisler measures are preserved under $\ast$ (cf.\ Lemma 4.23 in \cite{GHK} when working with the original definition). Here, we are interested in semigroup structures, thus we record the following lemma.

\begin{lemma}\label{lemma:ast_preserves_inv}
    If $\mu,\nu\in\mathfrak{M}^{\inv}_{\bar{n}}(\FC,M)$ are Borel $M$-definable,
    then $\mu\ast\nu\in\mathfrak{M}^{\inv}_{\bar{n}}(\FC,M)$.
\end{lemma}

\begin{proof}
We omit the proof as it is a straightforward application of the definitions
(e.g.\ see the proof of Lemma 4.23 in \cite{GHK}).
\end{proof}

The special case of $\ast$-product for types does not require Borel-definability due to the way the Morley product works for types. More precisely:

\begin{definition}\label{def:star.definition2}
    Let $p(\bar{x})\in S^{\inv}_{\bar{n}}(\FC,M)$ and $q(\bar{x})\in S_{\bar{n}}(\FC)$ (not necessarily $M$-invariant!). Then we set:
    $$(p\ast q)\big(\theta(\bar{x};\bar{b})\big):= \Big(p(\bar{x})\otimes h_{\bar{b}}\big(q(\bar{x})\big)\Big)(\theta(\bar{x};\bar{y})\big),$$
    where $\bar{b}\in \FC^{\bar{y}}$ and $\varphi(\bar{x};\bar{y})\in\CL_{\bar{x};\bar{y}}$ (cf.\ Definition 4.28 in \cite{GHK}).
\end{definition}

Let us note that $h_{\bar{b}}(q)$ is not a global $M$-invariant type (the usual situation in the definition of the Morley product for types), but the above definition still makes sense due to the $M$-invariance of $p$.

If $p$ is additionally Borel $M$-definable, then $\delta_{p\ast q}=\delta_p\ast\delta_q$ (the $\ast$-product for types coincides with the $\ast$-product for measures).
Similarly, to the measure case (Lemma~\ref{lemma:ast_preserves_inv}), we have that
$p\ast q\in S^{\inv}_{\bar{n}}(\FC,M)$ provided $p(\bar{x}),q(\bar{x})\in S^{\inv}_{\bar{n}}(\FC,M)$.

We also have an alternative definition of the $\ast$-product when the right argument is a type:

\begin{proposition}\label{prop: alter_ast}
    Let $\mu\in\mathfrak{M}^{\inv}_{\bar{x}}(\FC,M)$ be Borel $M$-definable, $p(\bar{x})\in S^{\inv}_{\bar{n}}(\FC,M)$ (not necessarily Borel $M$-definable)
and let $q(\bar{x})\in S_{\bar{n}}(\FC)$ with $q=\tp(\tau(\bar{n})/\FC)$ for some
$\tau\in\aut(\FC')$.
Then
\begin{align*}
    &&\mu\ast q&=\tau(\mu|_{\FC'})|_{\FC},\textrm{ and}&\textrm{(Def.~\ref{def:star.definition})}\\
    &&p\ast q&=\tau(p|_{\FC'})|_{\FC}&\textrm{(Def.~\ref{def:star.definition2})}
\end{align*}
(where $\mu|_{\FC'}$ and $p|_{\FC'}$ stand for the unique $M$-invariant extensions over $\FC'$).
\end{proposition}

\begin{proof}
    For $\mu$, fix
    $\theta(\bar{x};\bar{b})\in \CL$ and $\bar{b}\in \FC^{\bar{y}}$. By Definition~\ref{def:star.definition} we obtain 
    $$(\mu\ast q)\big(\theta(\bar{x};\bar{b})\big) = \mu|_{\FC'} \Big( \theta\big(\bar{x};\tau^{-1}(\bar{b})\big)\Big)=\big( \tau(\mu|_{\FC'})\big)\big( \theta(\bar{x};\bar{b})\big).$$
    For $p$, we can argue analogously (see also \cite[Proposition 4.29]{GHK}).
\end{proof}

The following observation follows immediately.

\begin{cor}\label{rem: aut_is_*}
    For every Borel $M$-definable $\mu\in\mathfrak{M}^{\inv}_{\bar{x}}(\FC,M)$,
    every $p\in S^{\inv}_{\bar{n}}(\FC,M)$ and every $\sigma\in \Aut(\FC)$, we have
    \begin{align*}
    &&\sigma\cdot \mu&=\mu*\tp(\sigma(\bar n)/\FC),\textrm{ and}&\textrm{(Def.~\ref{def:star.definition})}\\
    &&\sigma\cdot p&=p*\tp(\sigma(\bar n)/\FC).&\textrm{(Def.~\ref{def:star.definition2})}
    \end{align*}
\end{cor}

\begin{lemma}\label{lemma:ast_semigroup_for_types}
The space $(S^{\inv}_{\bar{n}}(\FC,M),\ast)$ is a Hausdorff compact left-continuous semigroup.
In the case of $\bar{n}=\bar{m}$, the type $\tp(\bar{m}/\FC)$ is its neutral element.
\end{lemma}

\begin{proof}
    Clearly, $S^{\inv}_{\bar{n}}(\FC,M)$ is Hausdorff and compact as a closed subset of $S_{\bar{x}}(\FC)$.
    Left-continuity of $\ast$ follows by left-continuity of the Morley product. Alternatively, one can argue as follows. Let $q,p,p_i\in S^{\inv}_{\bar{n}}(\FC,M)$, where $i\in I$, and $q=\tp(\tau(\bar{n})/\FC)$ for some $\tau\in\aut(\FC')$. Assume that $(p_i)_{i\in I}$ converges to $p$ and assume that $p\ast q\in [\theta(\bar{x};\bar{b})]$ for some $\theta(\bar{x};\bar{y})\in\CL$ and $\bar{b}\in \FC^{\bar{y}}$.
    The last thing is equivalent to $\theta(\bar{x};\tau^{-1}(\bar{b}))\in p|_{\FC'}$.

    Let $\bar{y}_0$ be a finite subtuple of $\bar{y}$, such that all the variables from the tuple $\bar{y}$ appearing in $\theta$ occur already among $\bar{y}_0$. Let $\bar{b}_0$ be the restriction of the tuple $\bar{b}$ to the variables $\bar{y}_0\subseteq\bar{y}$, i.e.\ $\bar{b}_0:=\bar{b}|_{\bar{y}_0}$. 
    By a slight abuse of notation,
    we can write $\theta(\bar{x};\bar{b}_0)$ in the place of $\theta(\bar{x};\bar{b})$. Now, let $\bar{d}_0\equiv_M\tau^{-1}(\bar{b}_0)$ be a tuple from $\FC$.
    Because $p$ is $M$-invariant, we have that $\theta(\bar{x};\bar{d}_0)\in p$.
    This means that there exists $i_0\in I$ such that for every $i\geqslant i_0$ we have
    $\theta(\bar{x};\bar{d}_0)\in p_i$.
    Hence $p_i\ast q\in [\theta(\bar{x};\bar{b})]$ for every $i\geqslant i_0$, and so $\lim_i p_i\ast q\in [\theta(\bar{x};\bar{b})]$.

    The associativity of $\ast$ follows by the same argument as in the proof of Proposition 4.33 of \cite{GHK}.
    Checking that $\tp(\bar{m}/\FC)$ is the neutral element is easy and so we skip it.
\end{proof}

\begin{remark}\label{rem: old star}
Let $\mu\in\mathfrak{M}^{\inv}_{\bar{x}}(\FC,M)$ be Borel $M$-definable
and let $\nu\in\mathfrak{M}_{\bar{n}}(\FC)$.
For every $\theta(\bar{x};\bar{y})\in\CL$ and $\bar{b}\in\FC^{\bar{y}}$, we have
$$(\mu\ast\nu)(\theta(\bar{x};\bar{b}))=\int_{S_{\bar{n}}(\FC)}(\mu\ast q)(\theta(\bar{x};\bar{b}))\,d\nu(q).$$
\end{remark}

\begin{proof}
    Simply note that
    $(\mu\ast q)(\theta(\bar{x};\bar{b}))= \big(F^{\theta(\bar{x};\bar{y})}_{\mu_{\bar{x}}}\circ\res_M\circ h_{\bar{b}}\big)(q)$ and so:
$$(\mu\ast\nu)(\theta(\bar{x};\bar{b}))=\int\limits_{q\in S_{\bar{n}}(\FC)}\big(F^{\theta(\bar{x};\bar{y})}_{\mu_{\bar{x}}}\circ\res_M\circ h_{\bar{b}}(q)\big)\,d\nu_{\bar{y}}
=\int\limits_{q\in S_{\bar{n}}(\FC)}(\mu\ast q)(\theta(\bar{x};\bar{b}))\,d\nu(q).$$
\end{proof}

In the following fact, the key assumption is that we consider a partial type in a small tuple of variables (i.e.\ in $\bar{x}'$):

\begin{fact}[improved version of Lemma 6.5 from \cite{GHK}]\label{rem: better 6.3}
Consider a partial $\emptyset$-type $F(\bar{x}';\bar{y}')\vdash\bar{x}'\equiv_{\emptyset}\bar{y}'$
and let $\mu \in \mathfrak{M}^{\inv}_{\bar{n}}(\FC,M)$
be Borel $M$-definable.
The following are equivalent:
\begin{enumerate}
    \item $G_{F,\FC}\cdot \mu =\{\mu\}$,
    \item $\mu \ast \big(S_{F(\bar{x}';\bar{m})}(\FC)\big)=\{\mu\}$,
    \item $\mu\ast\{\tp(\sigma(\bar{n})/\FC)\;:\;\sigma\in G_{F,\FC}\}=\{\mu\}$.
\end{enumerate}
\end{fact}
\begin{proof}
    For (1)$\Rightarrow$(2), we follow the argument as in the proof of Lemma 6.5 in \cite{GHK}.
    Then, (2)$\Rightarrow$(3) is obvious. To obtain (3)$\Rightarrow$(1), we use Corollary~\ref{rem: aut_is_*}.
\end{proof}

The following class of measures was introduced in 
\cite{HruPiSi13} and generalizes dfs measures (definable and finitely satisfiable). Moreover, for types, the following property coincides with generic stability \cite[Proposition 3.2]{ConGann20}.

\begin{definition}\label{def:fim}
    Let $\mu \in \mathfrak{M}_{\bar{x}}^{\inv}(\FC,M)$ be Borel-definable. 
We say that $\mu$ is a \emph{fim measure} over $M$ (a \emph{frequency interpretation measure} over $M$) if for any finite $\bar{x}' \subseteq \bar{x}$ and $\CL$-formula $\varphi(\bar{x}';\bar{y})$ there exists a sequence of $\mathcal{L}(M)$-formulas $(\theta_n(\bar{x}_1,\ldots, \bar{x}_n))_{1\leqslant n<\omega}$ such that $|\bar{x}_i| = |\bar{x}'|$ and:
    \begin{enumerate}
        \item for any $\epsilon>0$, there exists some integer $n_\epsilon$ such that 
        for every $n \geqslant n_\epsilon$ and every $\bar{a} = (\bar{a}_1,\dots,\bar{a}_n) \in \FC^{(\bar{x}_1,\ldots,\bar{x}_n)}$
        with $\models\theta_n(\bar{a})$ we have
        $$\sup_{b\in \FC^{\bar{y}}}|\Av(\bar{a})(\varphi(\bar{x}';\bar{b}))-\mu(\varphi(\bar{x}';\bar{b}))|<\epsilon,$$
        \item $\lim_{n\to\infty}\mu^{(n)}(\theta_n(\bar{x}_1,\ldots,\bar{x}_n))=1$.
    \end{enumerate}
We say that a fim measure $\mu \in \mathfrak{M}_{\bar{x}}^{\inv}(\FC,M)$ is \emph{super-fim} over $M$ if $\mu^{(n)}$ is fim for every $n \geq 1$. 
\end{definition}

Note that super-fim coincides with fim for Keisler measures in NIP theories (\cite{HruPiSi13}) and for types in NTP2 theories (\cite{ConGanHan23}).

\begin{remark}\label{rem: superfim}
    If $\mu\in\mathfrak{M}_{\bar{n}}(\FC)$ is fim over $M$ [superfim over $M$]
    then $\mu|_{\FC'}$ is fim over $M$ [superfim over $M$].
\end{remark}
\begin{proof}
    This is clear by the definition: $\CL(M)$-formulas witnessing fim over $M$ for $\mu^{(n)}$,
    witness fim over for the extension $\mu^{(n)}|_{\FC'}$,
    where $n<\omega$.
\end{proof}

\begin{proposition}\label{rem: inverse of star-product}
    Let $\mu,\nu\in\mathfrak{M}^{\inv}_{\bar{c}}(\FC,M)$ be such that
    $\mu$ is fim over $M$ and $\nu^{-1}$ is Borel $M$-definable.
    Then
    $$\mu\ast\nu=(\nu^{-1}\ast\mu^{-1})^{-1}.$$
\end{proposition}

\begin{proof}
    First, we note that $\mu_{\bar{x}}=  \big( (\mu_{\bar{x}})^{-1}_{\bar{y}}\big)^{-1}_{\bar{x}}$.
    We already defined $(\mu_{\bar{x}})^{-1}_{\bar{y}}$ as the pushforward $(h_{\bar{c}})_{\ast}\mu$ along
    $h_{\bar{c}}\colon S_{\bar{c}}(\FC)\to S_{\bar{y}}(\FC)$ (taking types in variable $\bar{x}$ to types in variable $\bar{y}$).
    Then, $\big( (\mu_{\bar{x}})^{-1}_{\bar{y}}\big)^{-1}_{\bar{x}}$ is defined in the same way after exchanging the roles of $\bar{x}$ and $\bar{y}$, i.e.\ 
    $\big( (\mu_{\bar{x}})^{-1}_{\bar{y}}\big)^{-1}_{\bar{x}}$ is pushforward $(h_{\bar{c}}')_{\ast} (\mu^{-1})$ along $h_{\bar{c}}'\colon S_{\bar{c}}(\FC)\to S_{\bar{x}}(\FC)$ (taking types in variable $\bar{y}$ to types in variable $\bar{x}$), given by $h'_{\bar{c}}\big(\tp(\sigma(\bar{c}/\FC))\big)=\tp\big(\sigma^{-1}(\bar{c})/\FC\big)$, where $\sigma\in\aut(\FC')$. Of course, $h_{\bar{c}}\circ h'_{\bar{c}}=\id_{S_{\bar{c}}(\FC)}$ and so $\mu_{\bar{x}}=  \big( (\mu_{\bar{x}})^{-1}_{\bar{y}}\big)^{-1}_{\bar{x}}$.

    Let $\theta(\bar{x};\bar{y})\in\CL$. We use Theorem 5.16 of \cite{InTheWild} (i.e.\ commutativity of fim measures in the Morley product):
    \begin{IEEEeqnarray*}{rCl}
        \big(\mu_{\bar{x}}\ast\nu_{\bar{x}}\big)\big(\theta(\bar{x};\bar{c})\big) &=& \Big( \mu_{\bar{x}}\otimes \big( (h_{\bar{c}})_\ast(\nu_{\bar{x}})\big)_{\bar{y}}\Big)\big(\theta(\bar{x};\bar{y})\big) \\
        &=& (\mu_{\bar{x}}\otimes \nu^{-1}_{\bar{y}})\big(\theta(\bar{x};\bar{y})\big) \\
        &=& (\nu^{-1}_{\bar{y}}\otimes \mu_{\bar{x}})\big(\theta(\bar{x};\bar{y})\big) \\
        &=& \Big(\nu^{-1}_{\bar{y}}\otimes \big((\mu_{\bar{x}})^{-1}_{\bar{y}}\big)^{-1}_{\bar{x}} \Big)\big(\theta(\bar{x};\bar{y})\big) \\
        &=& \big(\nu^{-1}_{\bar{y}}\ast \mu^{-1}_{\bar{y}}\big)\big(\theta(\bar{c};\bar{y})\big) \\
        &=& \big(\nu^{-1}_{\bar{y}}\ast \mu^{-1}_{\bar{y}}\big)^{-1}\big(\theta(\bar{x};\bar{c})\big) 
    \end{IEEEeqnarray*}
\end{proof}

\section{Idempotent Measure Conjecture}
\label{sec: ideals and stabilizers}
Recall that a relatively $\bar{m}$-type-definable over $M$
subgroup $G$ of $\aut(\FC)$, i.e.\ 
$G=\{\sigma\in\aut(\FC) : \;  \models\pi(\sigma(\bar{m});\bar{m})\}$
for some partial type $\pi(\bar{x}';\bar{y}')$ over $\emptyset$,
 is (left) fim (over $M$) if there exists a (left) $G$-invariant fim measure in $\mathfrak{M}^{\inv}_{\pi(\bar{x};\bar{m})}(\FC,M)$.
One of the main goals of \cite{GHK} (and also of \cite{CGK}) is to study the following conjecture (where $\bar{n}=\bar{m}$ and $\bar{x}'=\bar{x}$):

\begin{conj}[Idempotent Measure Conjecture]\label{conjecture: main conjecture}
Let $\mu \in \mathfrak{M}^{\df}_{\bar m}(\FC,M)$ be fim over $M$.
We know from \cite[Lemma 2.26]{GHK} that $\stab_l(\mu)=G_{\pi,\FC}$ for some partial type $\pi(\bar x,\bar y)\vdash\bar{x}\equiv_{\emptyset}\bar{y}$.
Then the following are equivalent:
\begin{enumerate}
\item $\mu$ is an idempotent.

\item $\mu$ is the unique (left) $G_{\pi,\FC}$-invariant measure in $\mathfrak{M}^{\inv}_{\pi(\bar{x};\bar{m})}(\FC,M)$.
\end{enumerate}
In particular, there is a correspondence between idempotent fim measures in $\mathfrak{M}^{\inv}_{\bar m}(\FC,M)$ and relatively $\bar m$-type-definable over $M$ fim subgroups of $\aut(\FC)$.
\end{conj}

The conjecture holds in all stable theories \cite[Corollary 8.25]{GHK} and for types (i.e.\ Dirac measures) in all rosy theories (in particular, in all simple and in all o-minimal theories) \cite[Corollary 6.21]{GHK}.
We already know that the implication (2)$\Rightarrow$(1) follows easily even without the fim assumption \cite[Corollary 6.7]{GHK}.
The difficult part is the implication (1)$\Rightarrow$(2).
In (2), uniqueness follows automatically from the assumption that $\mu$ is fim over $M$
(in an arbitrary theory for fim types and for superfim measures, in particular,
in NIP for measures, see \cite[Proposition 6.10]{GHK} and \cite[Corollary 6.13]{GHK}).
Therefore the main task in proving the above conjecture is to show that for idempotent fim measures we have $\supp(\mu)\subseteq[\pi(\bar{x};\bar{m})]$.
We will investigate this here
by looking at minimal left ideals.

Recall that
$$\ast\colon S^{\inv}_{\bar{n}}(\FC,M)\times S^{\inv}_{\bar{n}}(\FC,M)\to S^{\inv}_{\bar{n}}(\FC,M)$$
defines a Hausdorff compact left-continuous semigroup.
We can apply classical results, e.g.\ Fact A.8 from \cite{rzepecki2018},
to note that 
there exist minimal left ideals in $(S^{\inv}_{\bar{n}}(\FC,M),\ast)$, each of them is closed and contains an idempotent.

In the case of amenable NIP theories, the entire semigroup $(S^{\inv}_{\bar{n}}(\FC,M),\ast)$ contains a unique minimal left ideal, namely the ideal of types which do not fork over $\emptyset$ (we will understand more of the structure of this ideal later on, cf.\ Proposition~\ref{prop: support of f-generics} and Theorem~\ref{thm: fGen and injectivity on Ellis groups}).
For the general case (i.e.\ 
a theory that may be non-amenable or non-NIP),
we look for unique minimal left ideals not in the entire semigroup $(S^{\inv}_{\bar{n}}(\FC,M),\ast)$, but in subsemigroups given by supports of idempotent Keisler measures.

\subsection{Supports and stabilizers}
Let us fix $\mu\in\mathfrak{M}^{\df}_{\bar{n}}(\FC,M)$
with $\supp(\mu)\subseteq S^{\inv}_{\bar{n}}(\FC,M)$ (e.g.\ $\mu$ is fim over $M$). We set:
\begin{IEEEeqnarray*}{rCl}
S &:=& \supp(\mu), \\
\stab_l(\mu) &:=&
\{\sigma\in\aut(\FC)\;:\;\sigma\cdot\mu=\mu\},\\
\stab_r(\mu) &:=&
\{q\in S_{\bar{n}}(\FC)\;:\;\mu\ast q=\mu\}.
\end{IEEEeqnarray*}

\begin{remark}
    $\stab_r(\mu)$ is a closed subset of $S_{\bar{n}}(\FC)$.
\end{remark}

\begin{proof}
    Because $\mu$ is $M$-definable,
    we note that $(\mu\ast-)\colon S_{\bar{n}}(\FC)\to\mathfrak{M}_{\bar{n}}(\FC)$ is continuous. Indeed, if $q\in S_{\bar{n}}(\FC)$, $\theta(\bar{x};\bar{y})\in\CL$,
    $\bar{b}\in\FC^{\bar{y}}$
    and $r,s\in\Rr$ then
    \begin{IEEEeqnarray*}{rCl}
    (\mu\ast q)\big(\theta(\bar{x};\bar{b})\big)\in(r,s) &\iff &
    \big( F^{\theta(\bar{x};\bar{y})}_{\mu}\circ h_{\bar{b}}\big)(q)\in (r,s),
    \end{IEEEeqnarray*}
    and both $F^{\theta(\bar{x};\bar{y})}_\mu$ and $h_{\bar{b}}$ are continuous. It follows immediately that $\stab_r(\mu)$ is closed.
\end{proof}

\begin{proposition}\label{rem: idempotent support 1}
If, in addition, $\mu\ast\mu=\mu$, then
$(S,\ast)$ is a Hausdorff compact left-continuous semigroup with no proper closed two-sided ideals.
\end{proposition}

\begin{proof}
We notice that the proof of Proposition 8.2 from \cite{GHK} works well in our situation (i.e.\ with the tuple $\bar{n}$ in the place of tuple $\bar{m}$), and so we have that $S\ast S\subseteq S$. Thus clearly $(S,\ast)$ is a Hausdorff compact left-continuous semigroup.

The proof of Theorem 8.8 from \cite{GHK} works in our situation as well, and so we conclude that $(S,\ast)$ has no proper closed two-sided ideals (i.e.\ $S$ is simple).
\end{proof}

\begin{remark}\label{rem:left_simple_min_ideal}
    Assume that $\mu\ast\mu=\mu$.
    By the preceding proposition, it follows that $(S,*)$ is left-simple if and only if it has a unique minimal left ideal (for the nontrivial direction, notice that if the minimal left ideal is unique, it is also a right ideal).
\end{remark}

\begin{fact}
By Lemma 2.26 from \cite{GHK} (in fact, by its slight modification taking into account that $\bar{x}$ might be a long tuple, different from $\bar{x}'$) there is a partial $\emptyset$-type $\pi(\bar{x}';\bar{y}')\vdash\bar{x}'\equiv_{\emptyset}\bar{y}'$ (recall that $\bar{x}'$ corresponds to the tuple $\bar{m}$ and $\bar{y}'$ is its copy)
such that
$$\stab_l(\mu)=G_{\pi,\FC}=\{\sigma\in\aut(\FC)\;:\;\models\pi(\sigma(\bar{m});\bar{m})\}.$$
\end{fact}

Let us fix partial $\emptyset$-type $\pi(\bar{x}';\bar{y}')$
as in the above fact.

\begin{remark}\label{rem: stab and superfim}
    We have that $\stab_l(\mu|_{\FC'})=G_{\pi,\FC'}=\{\sigma\in\aut(\FC')\;:\;\models\pi(\sigma(\bar{m});\bar{m})\}$.
\end{remark}

\begin{proposition}\label{prop: left-right stab}
    $[\pi(\bar{x}';\bar{m})]\subseteq\stab_r(\mu)$.
\end{proposition}

\begin{proof}
Take $q=\tp(\sigma(\bar{n})/\FC)\in [\pi(\bar{x}';\bar{m})]$ with $\sigma\in\aut(\FC')$.
Then $\sigma\in G_{\pi,\FC'}$ and so $\sigma\in\stab_l(\mu|_{\FC'})$.
For every $\CL$-formula $\theta(\bar{x};\bar{y})$ and every parameter $\bar{b}\in \FC^{\bar{y}}$ we have
\begin{IEEEeqnarray*}{rCl}
\Big(\mu\ast\tp\big(\sigma(\bar{n})/\FC\big)\Big)\big(\theta(\bar{x};\bar{b})\big) &=& \mu|_{\FC'}\big(\theta(\bar{x};\sigma^{-1}(\bar{b}))\big) \\
&=& (\sigma\cdot \mu|_{\FC'})\big(\theta(\bar{x};\bar{b})\big) \\
&=& \mu\big(\theta(\bar{x};\bar{b})\big)
\end{IEEEeqnarray*}
and we see that $\mu\ast q=\mu$, i.e.\ $q\in\stab_r(\mu)$.
\end{proof}

\begin{lemma}\label{lemma: one minimal}
    If $\mu\ast\mu=\mu$
    and $L$ is a minimal left ideal of $S=\supp(\mu)$. Then:
    \begin{itemize}
        \item
        for each $q\in L$, we have $\supp(\mu*q)\subseteq L$,
        \item
        if $L\cap \stab_r(\mu)\neq\emptyset$, then $S=L$.
    \end{itemize}
\end{lemma}

\begin{proof}
    Let $L$ be a minimal left ideal of $S$ and let $q\in L$.
    Then $S\ast q=L$. By Proposition 8.3 from \cite{GHK} (again, its slight modification to work with possibly longer tuple $\bar{n}$ in the place of $\bar{m}$), 
    we have that $\supp(\mu\ast q)$ is contained in the closure of $S \ast q$, which is contained in $L$.

    Now, if $q\in L\cap \stab_r(\mu)$, then we have immediately that
    \[S=\supp(\mu)=\supp(\mu*q)\subseteq L\subseteq S.\qedhere\]
\end{proof}

Theorem 4.1 from \cite{Pym62} states that a measure on a left-simple kernel (a type of topological semigroup) is idempotent if and only if it is right invariant over its support. We adapted this idea in the following theorem, where the aforementioned properties are placed in a slightly different order:

\begin{theorem}\label{thm: minimal support vs stabilizer}
    Let $\mu\in\mathfrak{M}^{\df}_{\bar{n}}(\FC,M)$ be idempotent with $\supp(\mu)\subseteq S^{\inv}_{\bar{n}}(\FC,M)$.
    The following are equivalent:
    \begin{enumerate}
        \item $\supp(\mu)\subseteq \stab_r(\mu)$,
        \item $\supp(\mu)$ is left-simple, equivalently (by Remark~\ref{rem:left_simple_min_ideal}) it has a unique minimal left ideal $L$ (and then $\supp(\mu)=L$).
    \end{enumerate}
\end{theorem}

\begin{proof}
    (1) implies (2) immediately by Lemma~\ref{lemma: one minimal}.

    Now, assume (2).
    Consider a formula $\theta(\bar{x};\bar{y})$ and a parameter $\bar{b}\in\FC^{\bar{y}}$. The following function
   $$D:S\ni q\mapsto (\mu\ast q)(\theta(\bar{x};\bar{b}))$$
   is continuous. 
    Assume that $D$ achieves its maximum at $q_0\in S$. Then, by Proposition 8.7 from \cite{GHK} (more precisely, by its variant for tuple $\bar{n}$), we have that $D(p\ast q_0)=D(q_0)$ for every $p\in S$. By (2) it follows that $S=S*q_0$, so $D$ is constant on $S$.

    Fix $q\in S$ and note that
    \begin{IEEEeqnarray*}{rCl}
    \mu\big(\theta(\bar{x};\bar{b})\big) &=& (\mu\ast\mu)\big(\theta(\bar{x};\bar{b})\big) = \int\limits_{p\in S}(\mu\ast p)\big(\theta(\bar{x};\bar{b})\big)d\mu(p) \\
    &=& \int\limits_{p\in S} D(p)\,d\mu(p) = D(q)=(\mu\ast q)\big(\theta(\bar{x};\bar{b})\big).
    \end{IEEEeqnarray*}
    Because $\theta(\bar{x};\bar{b})$ was arbitrary,
    we see that $\mu\ast q=\mu$ for every $q\in S$.
    Hence $\supp(\mu)=S\subseteq\stab_r(\mu)$.
\end{proof}

Recall that to prove Conjecture~\ref{conjecture: main conjecture},
we need to show that $\supp(\mu)\subseteq[\pi(\bar{x}';\bar{m})]$.
By combining Proposition~\ref{prop: left-right stab} with Theorem~\ref{thm: minimal support vs stabilizer}, we have the following picture
$$\xymatrix{ & \supp(\mu) \ar@{--}[d]_-{\subseteq}^-{\iff\text{ left-simple}}\\
[\pi(\bar{x}';\bar{m})] \ar@{-}[r]^-{\subseteq} & \stab_r(\mu)}$$
yielding as a corollary:
\begin{cor}
    \label{cor:supp_cont_pi_ls}
    If $\supp(\mu)\subseteq[\pi(\bar{x}';\bar{m})]$
    (i.e.\ if $\mu$ is support transitive in the terminology of \cite[Definition 3.46]{CGK}), then $\supp(\mu)$ is left-simple. 
\end{cor}

Therefore, we consider replacing the original first point in Conjecture~\ref{conjecture: main conjecture} with the following:
\begin{enumerate}
    \item $\mu$ is idempotent and $\supp(\mu)$ is left-simple.
\end{enumerate}

\subsection{Restrictions and summary}
In this subsection we focus on Conjecture~\ref{conjecture: main conjecture}, thus we ease the notation by setting $\bar{n}=\bar{m}$ (and so $\bar{x}=\bar{x}'$).
Recall that $\mu\in\mathfrak{M}^{\df}_{\bar{m}}(\FC,M)$ and $\supp(\mu)\subseteq S^{\inv}_{\bar{m}}(\FC,M)$.

We introduce the following property:
\begin{equation}
    \label{eq:restricted_lascar_inv}
    \tag{$\diamondsuit$}
    \big(\forall\sigma\in\aut(\FC'),\,\tp(\sigma(\bar{m})/\FC)\in\supp(\mu)\big)
    \big(\mu|_{\FC'}\text{ is }\sigma^{-1}[M]\text{-invariant}\big)
\end{equation}

\begin{remark}\label{rem: diamond and Lascar}
    A standard argument shows that if $\mu$ is Lascar-invariant then so is $\mu|_{\FC'}$, which trivially implies \eqref{eq:restricted_lascar_inv} for $\mu$.
\end{remark}

\begin{lemma}\label{lemma: diamond}
\begin{enumerate}
    \item
    If $\supp(\mu)\subseteq[\pi(\bar{x};\bar{m})]$ then \eqref{eq:restricted_lascar_inv}.

    \item
    Assume that $\mu$ is idempotent, superfim over $M$ with left-simple $\supp(\mu)$.
    Then the converse holds as well, i.e.
    $\supp(\mu)\subseteq[\pi(\bar{x};\bar{m})]$ and \eqref{eq:restricted_lascar_inv} are equivalent.
\end{enumerate}
\end{lemma}

\begin{proof}
    Let $\sigma\in\aut(\FC')$ be such that $\tp(\sigma(\bar{m})/\FC)\in\supp(\mu)$.

    First, assume $\supp(\mu)\subseteq[\pi(\bar{x};\bar{m})]$.
    This means that
    $\sigma\in G_{\pi,\FC'}$. Hence $\sigma\cdot\mu|_{\FC'}=\mu|_{\FC'}$ (by Remark~\ref{rem: stab and superfim}) and so $\sigma\cdot\mu|_{\FC'}$
    is $M$-invariant. Thus $\mu|_{\FC'}$ is $\sigma^{-1}[M]$-invariant.

    For the proof of the second point, assume \eqref{eq:restricted_lascar_inv}, in particular
    $\mu|_{\FC'}$ is $\sigma^{-1}[M]$-invariant.
    By Theorem~\ref{thm: minimal support vs stabilizer}, we know that $\supp(\mu)\subseteq\stab_r(\mu)$.
    If $\theta(\bar{x};\bar{y})\in\CL$ then, by Proposition~\ref{prop: alter_ast},
    \begin{IEEEeqnarray*}{rCl}
    \mu\big(\theta(\bar{x};\bar{m})\big) &=&
    \Big( \mu\ast \tp\big(\sigma(\bar{m})/\FC\big)\Big)\big(\theta(\bar{x};\bar{m})\big) \\
    &=& \mu|_{\FC'}\big(\theta(\bar{x};\sigma^{-1}(\bar{m}))\big)
    = (\sigma\cdot \mu|_{\FC'})\big(\theta(\bar{x};\bar{m})\big).
    \end{IEEEeqnarray*}
    Therefore $(\mu|_{\FC'})|_M=\mu|_M=(\sigma\cdot \mu|_{\FC'})|_M$.

Because $\mu|_{\FC'}$ is $\sigma^{-1}[M]$-invariant,
$\sigma\cdot\mu|_{\FC'}$ must be $M$-invariant.
Because $\sigma\cdot\mu|_{\FC'}$ is Borel definable (as a shift of a Borel $M$-definable measure) and $M$-invariant, $\sigma\cdot\mu|_{\FC'}$ is Borel $M$-definable by Corollary 2.2 from \cite{InTheWild}.

Because $\mu|_{\FC'}$ is superfim over $M$ by Remark~\ref{rem: superfim}
and Borel $M$-definable, and $\sigma\cdot\mu|_{\FC'}$ is Borel $M$-definable,
we may use Lemma 6.12 from \cite{GHK} to conclude that
$(\mu|_{\FC'})|_M=(\sigma\cdot \mu|_{\FC'})|_M$ implies $\mu|_{\FC'}=\sigma\cdot\mu|_{\FC'}$.
This means that $\sigma\in\stab_l(\mu|_{\FC'})=G_{\pi,\FC'}$.
Therefore $\tp(\sigma(\bar{m})/\FC)\in [\pi(\bar{x}';\bar{m})]$ as expected.
\end{proof}

\begin{lemma}\label{lemma: 4 equivalent conditions}
    Suppose $\mu\in\mathfrak{M}^{\df}_{\bar{m}}(\FC,M)$ is a superfim over $M$ and idempotent.
    We know that $\stab_l(\mu)=G_{\pi,\FC}$ for some partial type $\pi(\bar x,\bar y)\vdash\bar{x}\equiv_{\emptyset}\bar{y}$.
    Consider the following conditions:
    \begin{enumerate}
        \item
        $\mu$ satisfies \eqref{eq:restricted_lascar_inv},
        \item
        $\supp(\mu)$ is left-simple (equivalently, $\supp(\mu)\subseteq \stab_r(\mu)$),
        \item 
        $\supp(\mu)\subseteq[\pi(\bar{x};\bar{m})]$,
        \item
        $\mu$ is the unique (left) $G_{\pi,\FC}$-invariant measure in $\mathfrak{M}^{\Borel}_{\pi(\bar{x};\bar{m})}(\FC,M)$.
    \end{enumerate}
    Then (1)$\land$(2) is equivalent to each of (3) and (4).
\end{lemma}

\begin{proof}
Notice that (4) contains (3), while (3) implies (4), by \cite[Corollary 6.13]{GHK}.
The equivalence in (2) holds by Theorem~\ref{thm: minimal support vs stabilizer}.
To see that (3) implies (2) we combine Proposition~\ref{prop: left-right stab} with Theorem~\ref{thm: minimal support vs stabilizer}.
By Lemma~\ref{lemma: diamond}, we have that (3) implies (1),
and moreover (1)$\land$(2) implies (3).  Putting these implications together finishes the proof.
\end{proof}

\begin{theorem}\label{thm: weak conjecture}
Let $\mu \in \mathfrak{M}^{\df}_{\bar m}(\FC,M)$ be superfim over $M$
and let \eqref{eq:restricted_lascar_inv} hold.
We know from \cite[Lemma 2.26]{GHK}
that $\stab_l(\mu)=G_{\pi,\FC}$ for some partial type $\pi(\bar x,\bar y)\vdash\bar{x}\equiv_{\emptyset}\bar{y}$.
Then the following are equivalent:
\begin{enumerate}
\item $\mu$ is an idempotent and $\supp(\mu)$ is left-simple.

\item $\mu$ is the unique (left) $G_{\pi,\FC}$-invariant measure in $\mathfrak{M}^{\Borel}_{\pi(\bar{x};\bar{m})}(\FC,M)$.
\end{enumerate}
\end{theorem}

\begin{proof}
    As was already mentioned, (2) implies that $\mu$ is idempotent by \cite[Lemma 6.5 and Proposition 6.6]{GHK}. 
    The conclusion follows easily by Lemma~\ref{lemma: 4 equivalent conditions}.
\end{proof}

There are four differences between the statement of Theorem~\ref{thm: weak conjecture} and Conjecture~\ref{conjecture: main conjecture}:
\begin{enumerate}
    \item
    \label{it:borel_definable}
    Considering the space $\mathfrak{M}^{\Borel}_{\pi(\bar{x};\bar{m})}(\FC,M)$ instead of $\mathfrak{M}^{\inv}_{\pi(\bar{x};\bar{m})}(\FC,M)$ in the second point.
    \\ This technical assumption is natural outside of the NIP context.
    
    \item
    \label{it:superfim}
    The measure $\mu$ is superfim over $M$ instead of being fim over $M$.
    \\ In case of the NIP theories, the class of superfim measures
    coincides with the class of fim measures.
    Under NTP2, for types, the class of superfim (i.e.\ super-generically stable) types
    coincides with the class of fim (generically stable) types.
    Thus most of the special cases of Conjecture~\ref{conjecture: main conjecture} confirmed in \cite{CGK,GHK} (except for 
    the case of an abelian type-definable group, cf.\ \cite[Theorem 3.45]{CGK}
    and the case of types in rosy theories \cite[Corollary 6.21]{GHK}) fall into the superfim case.

    \item
    \label{it:diamond_hyp}
    Property \eqref{eq:restricted_lascar_inv} holds.
    \\ This is a new assumption.
    However, by Lemma~\ref{lemma: diamond}(1), we see that every (idempotent) measure for which Conjecture~\ref{conjecture: main conjecture} holds (and every so-called \emph{generically transitive measure}) must satisfy \eqref{eq:restricted_lascar_inv}.

    \item
    \label{it:lsimple_support}
    In the first point, the measure $\mu$ is assumed to have left-simple support.
    \\ See the discussion at the end of the previous subsection, where we argue that left-simplicity of $\supp(\mu)$ is necessary to have $\supp(\mu)\subseteq[\pi(\bar{x};\bar{m})]$, which is the main difficulty in showing that (1)$\Rightarrow$(2) in Conjecture~\ref{conjecture: main conjecture}. 
    One of the main cases for which Conjecture~\ref{conjecture: main conjecture} is confirmed is the case of stable theories, and in stable theories left-simplicity of $\supp(\mu)$ holds (cf.\ Proposition 8.18 in \cite{GHK}).
\end{enumerate}

\begin{question}\label{question: diamond}
    Let $\mu \in \mathfrak{M}^{\inv}_{\bar m}(\FC,M)$ be fim over $M$
    [superfim over $M$] and idempotent, and suppose that $\supp(\mu)$ is left-simple.
    Does \eqref{eq:restricted_lascar_inv} follow?
\end{question}
\noindent
Note that if the answer to Question~\ref{question: diamond} is positive, then in Theorem~\ref{thm: weak conjecture} we can drop the \eqref{eq:restricted_lascar_inv} hypothesis.

Now, let us see how the main confirmed cases of Conjecture~\ref{conjecture: main conjecture} for measures can be recovered using Theorem~\ref{thm: weak conjecture}.
\begin{itemize}
    \item
    \emph{Stable case}
    (cf.\ \cite[Theorem 5.8]{Artem_Kyle}, \cite[Proposition 3.55]{CGK} and \cite[Corollary 8.25]{GHK})
    \\
    In the stable case, all invariant measures over models are definable (in particular Borel definable), so \eqref{it:borel_definable} makes no difference.
    As mentioned earlier, superfim coincides with fim for stable theories (both for types and measures), so difference \eqref{it:superfim} vanishes.
    By Proposition 8.18 from \cite{GHK}, the distinction \eqref{it:lsimple_support} also disappears. Property \eqref{eq:restricted_lascar_inv} is new,
    but it was obtained in the proof of \cite[Lemma 8.21]{GHK} (\eqref{eq:restricted_lascar_inv} corresponds to $\sigma\cdot \hat \mu|_M$ not forking over $M$ there), 
    where Newelski's Group Chunk Theorem for automorphism groups was applied, cf.\ \cite[Theorem 7.25]{GHK}.

    \item
    \emph{NIP+KP-invariant case}
    (cf.\  \cite[Theorem 4.11]{Artem_Kyle2},  \cite[Corollary 6.32,
    Theorem 6.31]{GHK})
    \\ Here, we consider Conjecture~\ref{conjecture: main conjecture} with the additional assumption that $\mu$ is KP-invariant ($\autf_\KP(\FC)$-invariant).
    Again, by NIP, an $M$-invariant measure is Borel-definable (corresponding to \eqref{it:borel_definable}) and superfim coincides with fim (cf.\ \eqref{it:superfim}).
    Then, if $\mu$ is KP-invariant, it is Lascar-invariant and \eqref{eq:restricted_lascar_inv} follows (by Remark~\ref{rem: diamond and Lascar}, i.e.\ \eqref{it:diamond_hyp}).
    Moreover, we will see later that KP-invariance implies that $\supp(\mu)$ is left-simple (see the proof of Corollary~\ref{cor: conj when f-gen in support} and Corollary~\ref{cor: description of idempotent measures with f-generics}(3)), which corresponds to \eqref{it:lsimple_support}.
\end{itemize}

Now, we will restate Theorem~\ref{thm: weak conjecture} separately for the case of types, where we can make several simplifications.

\begin{theorem}\label{thm: weak conjecture_types}
Let $p \in S^{\inv}_{\bar m}(\FC,M)$ be generically stable over $M$
and let \eqref{eq:restricted_lascar_inv} hold (for $\mu=p$).
We know that $\stab(p)=G_{\pi,\FC}$ for some partial type $\pi(\bar x,\bar y)\vdash\bar{x}\equiv_{\emptyset}\bar{y}$.
Then the following are equivalent:
\begin{enumerate}
\item $p$ is an idempotent.

\item $p$ is the unique (left) $G_{\pi,\FC}$-invariant type in $S^{\inv}_{\pi(\bar{x};\bar{m})}(\FC,M)$.
\end{enumerate}
\end{theorem}

\begin{proof}
    Again, we have that (2)$\Rightarrow$(1) follows
    by \cite[Lemma 6.5]{GHK} and \cite[Proposition 6.6]{GHK}.

    Now, let us assume (1). Obviously, $\supp(\delta_p)$ is left-simple.
    We would like to use Lemma~\ref{lemma: diamond}, 
    but this would require $p$ to be
    superfim. However, we can follow the proof of Lemma~\ref{lemma: diamond}(2)
    to observe that in the case of types we do not need to assume that $p$ is superfim.
    Indeed, the only place where we need it is the use of \cite[Lemma 6.12]{GHK}
    and instead, we can use \cite[Proposition 2.1(iv)]{PiTa} which gives the same conclusion for generically stable types.
    The uniqueness in (2) follows by \cite[Proposition 6.10]{GHK}.
\end{proof}

\cite[Proposition 2.37]{CGK} and \cite[Corollary 6.21]{GHK} show that Conjecture~\ref{conjecture: main conjecture} holds for types in rosy theories.
By Lemma~\ref{lemma: diamond}, we have that every idempotent fim type in a rosy theory satisfies \eqref{eq:restricted_lascar_inv}. And again, showing property \eqref{eq:restricted_lascar_inv} was
crucial for the proofs of \cite[Proposition 2.37]{CGK} (see also Proposition 2.30 there) and  in \cite[Corollary 6.21]{GHK}.

\begin{question}\label{question: diamond2}
    Let $p\in S^{\inv}_{\bar m}(\FC,M)$ be idempotent and fim over $M$.
    Does \eqref{eq:restricted_lascar_inv} follow?
\end{question}

We see that
to show Conjecture~\ref{conjecture: main conjecture}
we need to study Questions~\ref{question: diamond} and~\ref{question: diamond2}.
Removing the assumption on left-simplicity in the first point of Theorem~\ref{thm: weak conjecture} might be possible due to the stability hidden behind the notion of a fim/superfim measure/type:

\begin{question}\label{question: left-simple}
    Let $\mu \in \mathfrak{M}^{\inv}_{\bar m}(\FC,M)$ be fim over $M$
    [superfim over $M$] and idempotent.
    Is $\supp(\mu)$ left-simple?
\end{question}

\section{F-generics}\label{sec: f-generics and idempotent measures}
In this section, we define the notion of f-generics.
After that, we observe good properties of the $\ast$-product when f-generics are involved. In NIP, this leads to results on the minimal ideals of the semigroup $(S^{\inv}_{\bar{n}}(\FC,M),\ast)$ 
and to a better understanding of the semigroup structure of supports of idempotent Keisler measures.

\subsection{f-generics and minimal ideals}\label{subsec: f-generics}
We begin by introducing
the notion of f-generics and showing that they form a minimal left ideal (provided they exist in an NIP theory $T$).

Recall that $M\preceq N\preceq\FC$, $\bar{m}\in M^{\bar{x}'}$ enumerates $M$,
$\bar{m}\subseteq\bar{n}\in N^{\bar{x}}$ enumerates $N$, and $\bar{x}'\subseteq\bar{x}$. Moreover, $\FC'\preceq\FC''$ 
are successively larger monster models
containing $\FC$.
We define the following map
$$\rho\colon S_{\bar{n}}(\FC)\to\gal_{\Las}(T)$$
$$\rho(\tp(\sigma(\bar{n})/\FC)):=\sigma^{-1}\autf_{\Las}(\FC')$$
and its variant, denoted $\rho'\colon S_{\bar{n}}(\FC')\to\gal_{\Las}(T)$, with $\FC'\preceq\FC''$ in the place of $\FC\preceq\FC'$. 
Note that $\rho\colon S_{\bar{n}}(\FC)\to\gal_{\KP}(T)$ for $T$ being $G$-compact.

Moreover, recall that $\stab_l(\mu):=\{\sigma\in\aut(\FC):\sigma_*\mu=\mu\}$, where $\mu\in\mathfrak{M}(\FC)$. In particular, $\stab_l(p)=\{\sigma\in\aut(\FC):\sigma\cdot p=p\}$ for $p\in S(\FC)$.

\begin{proposition}\label{prop: stab in Lascar}
    If $p$ is a global type extending the type of a model (e.g.\ $p(\bar{x})\in S_{\bar{n}}(\FC)$) then $\stab_l(p)\leqslant\autf_{\Las}(\FC)\leqslant\autf_{\KP}(\FC)$.
\end{proposition}
\begin{proof}
The natural surjection from types to the Galois group is $\Aut(\FC)$-equivariant, and the stabilizer of any point in $\gal(T)$ is $\autf_L(\FC)$.
\end{proof}

\begin{remark}\label{rem: 4 properties}
    Let $p(\bar{x})\in S_{\bar{n}}(\FC)$. Assuming NIP, the following are equivalent by  \cite[Fact 2.22]{HruKruPi} (based on \cite[Proposition 2.11]{Anand_Udi2011}):
    \begin{enumerate}
        \item $p$ does not fork over $\emptyset$,
    
        \item the orbit $\aut(\FC)\cdot p$ is bounded,
        
        \item $p$ is Lascar-invariant (i.e.\ $\autf_{L}(\FC)\leqslant\stab_l(p)$ \\ and so
        $\autf_{L}(\FC)=\stab_l(p)$ by Proposition~\ref{prop: stab in Lascar}),
        
        \item $p$ is KP-invariant (i.e.\ $\autf_{\KP}(\FC)\leqslant\stab_l(p)$ \\ and so
        $\autf_{\KP}(\FC)=\autf_{L}(\FC)=\stab_l(p)$ by Proposition~\ref{prop: stab in Lascar}).
    \end{enumerate}
    Without NIP, we have that (4)$\Rightarrow$(3)$\Leftrightarrow$(2)$\Rightarrow$(1). 
    Indeed, we have that (2)$\Leftrightarrow$(3) by \cite[Fact 4.4]{HruKruPi}, (4)$\Rightarrow$(3) is trivial, and (2)$\Rightarrow$(1) is standard.
    Note also that trivially (3)$\Leftrightarrow$(4) in $G$-compact theories, in particular (3)$\Leftrightarrow$(4) in amenable theories by the main result of \cite{HruKruPi} (i.e.\ amenable theories are $G$-compact).
\end{remark}

\noindent
One could try to use Proposition 4.5 from \cite{HruKruPi} to show that (2)$\Rightarrow$(4) in Remark~\ref{rem: 4 properties}, but for that we should know that $\stab_l(p)$ is a relatively type-definable subgroup of $\aut(\FC)$ (which is the case for definable types, cf.\ Lemma 2.26 in \cite{GHK}).

\begin{definition}
    \label{dfn:f-gen}
    We say that $p(\bar{x})\in S^{\inv}_{\bar{n}}(\FC,M)$
    is (left) f-generic if its ($\aut(\FC)$-)orbit is bounded, equivalently, its left stabilizer is equal to $\autf_L(\FC)$.
    Let $\fGen\subseteq S^{\inv}_{\bar{n}}(\FC,M)$ denote the set of all f-generic types in $S^{\inv}_{\bar{n}}(\FC,M)$.
\end{definition}

\begin{proposition}
    (NIP) A theory $T$ is amenable if and only if f-generics (in the sense of Definition~\ref{dfn:f-gen}) exist.
\end{proposition}

\begin{proof}
Simply combine Remark~\ref{rem: 4 properties} (or \cite[Fact 2.22]{HruKruPi})
and \cite[Fact 2.23]{HruKruPi}.
\end{proof}

\begin{proposition}\label{prop: support of f-generics}
    (NIP)
    If $\mu\in\mathfrak{M}_{\bar{n}}(\FC)$
    is an $\Aut(\FC)$-invariant measure, then every type in the support of $\mu$ is f-generic.
\end{proposition}
\begin{proof}
    Let $p(\bar{x})$ be a type in the support of $\mu$. Then $p$ does not fork over $\emptyset$ by a standard argument, and so by Remark~\ref{rem: 4 properties} it is f-generic.
\end{proof}

\begin{lemma}\label{lemma: f-generic extension}
    Let $p(\bar{x})\in S^{\inv}_{\bar{n}}(\FC,M)$ be an f-generic.
    Then $p|_{\FC'}$ is an f-generic, i.e.\ $\stab_l(p|_{\FC'})=\autf_{\Las}(\FC')$.
    Consequently, $\rho'$ is injective on the $\aut(\FC')$-orbit of $p|_{\FC'}$.
\end{lemma}

\begin{proof}
    It suffices to show that if $\bar b_1\equiv_{\Las}\bar b_2$, then $\varphi(\bar x,\bar b_1)\in p|_{\FC'}$ if and only if $\varphi(\bar x,\bar b_2)\in p|_{\FC'}$. Let $\bar c_1,\bar c_2\in \FC$ be such that $\bar b_j\equiv_M \bar c_j$ for $j=1,2$. 
    Then, $\bar c_1\equiv_{\Las} \bar b_1\equiv_{\Las} \bar b_2\equiv_{\Las} \bar c_2$, and since $p|_{\FC'}\supseteq p$ is $M$-invariant and $p$ is $\autf_{\Las}(\FC)$-invariant, we have that
    $$\varphi(\bar x,\bar b_1)\in p|_{\FC'}\Leftrightarrow
        \varphi(\bar x,\bar c_1)\in p \Leftrightarrow
        \varphi(\bar x,\bar c_2)\in p\Leftrightarrow
        \varphi(\bar x,\bar b_2)\in p|_{\FC'}$$

    To prove the injectivity, consider $\sigma',\tau'\in\aut(\FC')$ such that
    $$\rho'(p|_{\FC'})(\sigma')^{-1}\autf_{\Las}(\FC')=\rho'(\sigma'\cdot p|_{\FC'})=\rho'(\tau'\cdot p|_{\FC'})=\rho'(p|_{\FC'})(\tau')^{-1}\autf_{\Las}(\FC').$$
    Hence, $(\tau')^{-1}\sigma'\in\autf_{\Las}(\FC')=\stab_l(p|_{\FC'})$ and
    \[\tau'\cdot p|_{\FC'}= \tau'\cdot \big((\tau')^{-1}\sigma' \big)\cdot p|_{\FC'}=\sigma'\cdot p|_{\FC'}.\qedhere\]
\end{proof}

\begin{lemma}\label{lemma: cdot vs ast}
    Suppose $p(\bar{x})\in S^{\inv}_{\bar{n}}(\FC,M)$ is f-generic
    and that
    $\sigma\in \Aut(\FC)$. Let $q_0\coloneqq \tp(\sigma(\bar m)/M)$ and let $q$ be any
    extension of $q_0$ in $S_{\bar{n}}(\FC)$ (e.g.\ $q$ can be a coheir of $\tp(\sigma(\bar{n})/M)$). Then we have $\sigma\cdot p=p*q$.
\end{lemma}
\begin{proof}
    Let $\tau\in\Aut(\FC')$ be such that $q=\tp(\tau(\bar n)/\FC)$, and let $\sigma'\in\Aut(\FC')$ be an arbitrary extension of $\sigma$.

    Note that by the choice of $q_0$, the image of $\tau$ in $\gal_{\Las}(T)$
    (via the quotient map $\aut(\FC')\to\aut(\FC')/\autf_{\Las}(\FC')$)
    is the same as the image of $\sigma'$.
    Then
    $$\rho'(\tau\cdot p|_{\FC'})
    =\rho'(p|_{\FC'})\cdot \tau^{-1}\autf_{\Las}(\FC')
    =\rho(p)\cdot (\sigma')^{-1}\autf_{\Las}(\FC')
    =\rho'(\sigma'\cdot p|_{\FC'}).$$

    Since $\rho'$ is injective on the orbit of $p|_{\FC'}$
    and we have $\sigma'\cdot p|_{\FC'}=(\sigma\cdot p)|_{\FC'}$,
    it follows that $(\sigma\cdot p)|_{\FC'}=\tau\cdot p|_{\FC'}$.
    Hence
    $(\sigma\cdot p)|_{\FC'}=\tau\cdot p|_{\FC'}=(p*q)|_{\FC'}$.
\end{proof}

\begin{cor}\label{cor: orbit of f-generic}
    If $p(\bar{x})\in S^{\inv}_{\bar{n}}(\FC,M)$ is f-generic, then $\Aut(\FC)\cdot p= p*S^{\inv}_{\bar{n}}(\FC,M)$.
\end{cor}
\begin{proof}
    $\subseteq$ is immediate by Lemma~\ref{lemma: cdot vs ast}.

    For $\supseteq$, take any $q\in S^{\inv}_{\bar{n}}(\FC,M)$ and let $\sigma\in\Aut(\FC)$ be such that $\tp(\sigma(\bar m)/M)=q|_{\bar{x}',M}$. Then by Lemma~\ref{lemma: cdot vs ast}, $\sigma\cdot p=p*q$.
\end{proof}

\begin{cor}
    \label{cor: fgen_by_right}
    For $p\in  S^{\inv}_{\bar{n}}(\FC,M)$ ,
    the following are equivalent:
    \begin{enumerate}
        \item
        $p$ is f-generic,
        \item
        $p* S_{\bar{n}}(\FC)$ is bounded.
    \end{enumerate}
\end{cor}
\begin{proof}
    The first condition implies the second by Lemma~\ref{lemma: cdot vs ast}.
    The converse follows immediately from Corollary~\ref{rem: aut_is_*}.
\end{proof}

\begin{theorem}\label{thm: fGen and injectivity on Ellis groups}
    If $T$ is an amenable NIP theory then:
    \begin{enumerate}
        \item
        the set $\fGen$ is a two-sided ideal of $(S^{\inv}_{\bar{n}}(\FC,M),\ast)$,

        \item
        for each $p(\bar{x})\in \fGen$,
        we have that $p*S^{\inv}_{\bar{n}}(\FC,M)=\Aut(\FC)\cdot p$ is an Ellis group in $S^{\inv}_{\bar{n}}(\FC,M)$ (in particular, it is a minimal right ideal),

        \item
        $\fGen$ is the unique minimal left ideal in $(S^{\inv}_{\bar{n}}(\FC,M),\ast)$,

        \item
        the Ellis groups in $(S^{\inv}_{\bar{n}}(\FC,M),\ast)$ are exactly the $\aut(\FC)$-orbits of f-generic types,

        \item
        the map $\rho\colon S^{\inv}_{\bar{n}}(\FC,M)\to \gal_{\KP}(T)$ is injective on Ellis groups in $(S^{\inv}_{\bar{n}}(\FC,M),\ast)$.
    \end{enumerate}
\end{theorem}
\begin{proof}
    Being a left ideal in (1) follows by Corollary~\ref{cor: fgen_by_right},
    and being a right ideal follows by Corollary~\ref{cor: orbit of f-generic} and the definition of an f-generic.

    Now, we argue for (2).
    Take any f-generic $p(\bar{x})\in S^{\inv}_{\bar{n}}(\FC,M)$.
    Then by Corollary~\ref{cor: orbit of f-generic}, $\Aut(\FC)\cdot p$ is a right ideal in $(S^{\inv}_{\bar{n}}(\FC,M),\ast)$, so it contains a minimal idempotent $u=\sigma_0\cdot p$ (i.e.\ $u*S^{\inv}_{\bar{n}}(\FC,M)$ is a minimal right ideal). Then $u$ is also f-generic and $\Aut(\FC)\cdot p=\Aut(\FC)\cdot u$ and $u*S^{\inv}_{\bar{n}}(\FC,M)=p*S^{\inv}_{\bar{n}}(\FC,M)$, so without loss of generality $p=u$.

    Note that for each $q\in \Aut(\FC)\cdot u$, we have $q*u=q$.
    Indeed, let $\tau\in\Aut(\FC)$ be such that $\tp(\tau(\bar m)/M)\subseteq u$.
    By Lemma~\ref{lemma: cdot vs ast} we have
    $$\rho(u)=\rho(u\ast u)=\rho(\tau\cdot u)=\rho(u)\tau^{-1}\autf_{\KP}(\FC),$$
    hence $\tau\in\autf_{KP}(\FC)$.
    As $q$ is an f-generic, we have $q*u=\tau\cdot q$ (by Lemma~\ref{lemma: cdot vs ast}).
    Then, since $\tau\in \autf_{KP}(\FC)$ and $q$ is f-generic, $\tau\cdot q=q$.

    It follows that $\Aut(\FC)\cdot u= (\aut(\FC)\cdot u)\ast u =\big(u*S^{\inv}_{\bar{n}}(\FC,M)\big)*u$, which is an Ellis group (of the right ideal of all f-generic types in $S^{\inv}_{\bar{n}}(\FC,M)$), because $u$ is a minimal idempotent.

    For (3), note that by (2), Ellis groups are right ideals, so the minimal left ideal is unique and contained in $\fGen$ by (1).
    This minimal left ideal is a union of all Ellis groups.
    By (2), each f-generic type is in an Ellis group, hence in the minimal left ideal and we see that $\fGen$ must be equal to the minimal left ideal.

    The fourth condition follows easily from the first three points.
    The fifth follows from the fourth and the definition of f-generic.
\end{proof}

Let us note that Theorem~\ref{thm: fGen and injectivity on Ellis groups} and its proof are parallel to their analogues for definable groups by the second author and Kyle Gannon in \cite[Section 3]{GR25}. There is also a similar earlier result (also for definable groups) \cite[Lemma 2.3]{Anand2013}.

\subsection{Idempotent measures with f-generics}
In this subsection, we notice that Conjecture~\ref{conjecture: main conjecture} holds in amenable NIP theories provided there is an f-generic in the support of the considered Keisler measure.
Before that, we analyze the structure of the support of an idempotent measure which has an f-generic type in its support.

Let $\mu\in\mathfrak{M}^{\df}_{\bar{n}}(\FC,M)$ be idempotent with $\supp(\mu)\subseteq S^{\inv}_{\bar{n}}(\FC,M)$.
Let $S:=\supp(\mu)$, which is a compact Hausdorff left-continuous sub-semigroup in $(S^{\inv}_{\bar{n}}(\FC,M),\ast)$.
Recall that $\rho\colon S^{\inv}_{\bar{c}}(\FC,M)\to\gal_{\KP}(T)$ is given by
$\rho(\tp(\sigma(\bar{c})/\FC))=\sigma^{-1}\autf_{\KP}(\FC')$, where $\tp(\sigma(\bar{c})/\FC)\in S^{\inv}_{\bar{c}}(\FC,M)$ and $\sigma\in\aut(\FC')$.

\begin{cor}
    (NIP)
    Suppose that there is an f-generic type in $S$. Then the natural homomorphism $\rho|_S\colon S\to \gal_{KP}(T)$ is injective on Ellis groups.
\end{cor}
\begin{proof}
    Immediate by Proposition~\ref{prop: ideals vs semigroups}
    and Theorem~\ref{thm: fGen and injectivity on Ellis groups}.
\end{proof}

\begin{cor}\label{cor: description of idempotent measures with f-generics}
    (NIP)
    Suppose there is an f-generic type in $S$.
    Then:
    \begin{enumerate}
        \item $\mu$ is $\autf_{\KP}(\FC)$-invariant,

        \item
        \label{it:kernel_idempotents}
        the kernel of $\rho|_S$ consists exactly of the idempotents in $S$,
        \item
        $S$ is left-simple,
        \item
        \label{it:product_idemp}
        the product of idempotents in $S$ is idempotent,
        \item
        \label{it:structure_S}
        $S$ is of the form $G\times J$, where $G$ is a subgroup of $\gal_{KP}(T)$ and $J$ is a semigroup of idempotents (isomorphic to $J_1\times J_2$, where $J_1$ is a left semigroup and $J_2$ is a right semigroup).

        \item
        \label{it:support_transfer}
        $G=\rho[\supp(\mu)]=\supp(\rho_\ast\mu)$ and $(\rho_\ast\mu)|_G$ is the normalized Haar measure on the closed (in the logic topology) subgroup $G\leqslant\gal_{\KP}(T)$.
        Moreover, $\stab_l(\mu)\leqslant\aut(\FC)$ is the preimage of $G\leqslant\gal_{\KP}(T)$
        via the canonical quotient map $\aut(\FC)\to\aut(\FC)/\autf_{\KP}(\FC)$.

    \end{enumerate}
\end{cor}
\begin{proof}
    By Remark~\ref{rem: S in I}, if $S$ contains an f-generic type then $S\subseteq\fGen$.
    By NIP, $\mu$ is $M_0$-invariant for every small model $M_0$, i.e.\ $\mu$ is $\autf_{\Las}(\FC)$-invariant.
    Then $\mu$ is $\autf_{\KP}(\FC)$-invariant since $\autf_{\Las}(\FC)=\autf_{\KP}(\FC)$.

    For the proof of \eqref{it:kernel_idempotents}, suppose $p\in \ker \rho|_S$. Then since $S$ is simple, $p$ is in an Ellis group $pSp$. If $u\in pSp$ is idempotent, then $u\in \ker \rho|_S$, and since $\rho$ is injective on $pSp$, it follows that $p=u$. This gives us (1).

    \eqref{it:product_idemp} follows since all idempotents are in $\ker\rho|_S$ and $\ker\rho|_S$ is a subsemigroup.

    For \eqref{it:structure_S}, let $G\leqslant \gal_{KP}(T)$ be the image of $\rho|_S$ and let $J\coloneqq \ker\rho|_S$; then $J$ is the set of idempotents in $S$. It is easy to see that $S\cong G\times J$.

    For \eqref{it:support_transfer}, we need to note that $\supp(\rho_\ast\mu)=\rho[\supp(\mu)]$. This follows since $\rho$ is a continuous open map, $\rho^{-1}\rho[A]=\autf_{\KP}(\FC)\cdot A$ and $\mu$ is $\autf_{\KP}(\FC)$-invariant. Then we use Theorem 6.30 from \cite{GHK}.
\end{proof}

\begin{cor}\label{cor: conj when f-gen in support}
    (NIP)
    Conjecture~\ref{conjecture: main conjecture} holds for measures having an f-generic type in their support.
\end{cor}

\begin{proof}
    Let $\bar{n}=\bar{m}$ and consider an idempotent
    $\mu\in\mathfrak{M}^{\df}_{\bar{m}}(\FC,M)$ with $S=\supp(\mu)$ as above.
    If $S\cap\fGen\neq\emptyset$ then $\mu$ is $\autf_{\KP}(\FC)$-invariant 
    by Corollary~\ref{cor: description of idempotent measures with f-generics}(1).
    Then, we apply Theorem 6.31 from \cite{GHK}.
\end{proof}

\section{Decomposing invariant measures}
In this section, we assume that $T$ is 
NIP and we aim to describe $\aut(\FC)$-invariant measures by a decomposition into canonical invariant measures living on the Ellis groups of the minimal left ideal $\fGen$.
In other words, we will describe the ergodic Keisler measures, similarly to what was done in \cite{ArtemPierre} for the case of a definable group.
We adapted the proofs of Sections 3 and 4 of \cite{ArtemPierre} to work in the case of the automorphism group.

We start with setting the backstage.
Of course, we assume that there is at least one invariant Keisler measure, i.e.\ we assume that $T$ is an amenable NIP theory.
Recall that $M\preceq N\preceq\FC$ and $M$ is small. We also fix enumerations $\bar{m}\subseteq\bar{n}\subseteq\bar{c}$ of $M$, $N$ and $\FC$, respectively. Moreover, $\bar{x}'$ and $\bar{y}'$ are copies of a tuple of variables related to $\bar{m}$, and $\bar{x}$ and $\bar{y}$ are copies of a tuple of variables related to $\bar{n}$.

\subsection{Almost pullback}
Recall that we have the following map
$$\rho\colon S_{\bar{n}}(\FC)\to\gal_{\KP}(T)$$
$$\rho(\tp(\sigma(\bar{n})/\FC)):=\sigma^{-1}\autf_{\KP}(\FC').$$
Let $S:=S^{\inv}_{\bar{n}}(\FC,M)$ and let $I:=\fGen\subseteq S$.
For every $p\in I$, $\aut(\FC)\cdot p=p\ast S$ is an Ellis group of the minimal left ideal $I$ (by Theorem~\ref{thm: fGen and injectivity on Ellis groups}).
Moreover,
there is $u^2=u\in I$ such that $p\ast S=u\ast I$ and the restriction of $\rho$:
$$\rho|_{u\ast I}\colon u\ast I\to\gal_{\KP}(T)$$
is continuous, bijective homomorphism of  groups (note that $u*I$ may not be compact, so it does not immediately follow that $\rho|_{u\ast I}$ is a topological isomorphism).
Let $r_{\FC}\colon \aut(\FC')/\autf_{\KP}(\FC')\to\aut(\FC)/\autf_{\KP}(\FC)$ be the canonical isomorphism of Kim-Pillay groups.

\begin{lemma}\label{lemma: rho is Borel}
    The inverse of $\rho|_{u\ast I}$ is a Borel map.
    In particular $\rho|_{u\ast I}$ is a Borel group isomorphism (i.e.\ it is a Borel isomorphism of groups whose inverse is also Borel).
\end{lemma}

\begin{proof}
    To see that the inverse of $\rho|_{u\ast I}$ is a Borel map, we check that for every $\theta(\bar{x};\bar{y})$ and $\bar{b}\in\FC^{\bar{y}}$ the set
    $\rho\big[ [\theta(\bar{x};\bar{b})]\,\cap\,u\ast I\big]$ is a Borel subset of $\gal_{\KP}(T)$.

    Then (using the fact that $\rho$ is a homomorphism
    to get the second line
    and then Lemma~\ref{lemma: cdot vs ast} to get the last line),
    \begin{IEEEeqnarray*}{rCl}
        r_{\FC}\rho\big[ [\theta(\bar{x};\bar{b})]\,\cap\,u\ast I\big] &=&
        r_{\FC}\big\{ \rho(u\ast q)\;:\;q\in I,\,u\ast q\in[\theta(\bar{x};\bar{b})]\big\} \\
        &=& r_{\FC}\big\{\rho(q)\;:\; q\in I,\,u\ast q\in[\theta(\bar{x};\bar{b})]\big\} \\
        &=& \big\{r_{\FC}\rho(q)\;:\; q\in I,\,u\ast q\in[\theta(\bar{x};\bar{b})]\big\} \\
        &=& \big\{\sigma^{-1}\autf_{\KP}(\FC)\;:\;\sigma\in\aut(\FC),\, \sigma(u)\in[\theta(\bar{x};\bar{b})]\big\}
    \end{IEEEeqnarray*}
    and this set is Borel (even constructible) by the discussion after Fact 2.23 in \cite{HruKruPi}.
    As $r_{\FC}$ is a homeomorphism, we have that
    also $\rho\big[ [\theta(\bar{x};\bar{b})]\,\cap\,u\ast I\big]$ is Borel (even constructible).
\end{proof}

Let $h_{\FC}$ and $h_{\FC'}$ be the unique normalized Haar measures on
$$\aut(\FC)/\autf_{\KP}(\FC)\cong \aut(\FC')/\autf_{\KP}(\FC'),$$
respectively.
The set in the last exposed line in the above proof has the same $h_{\FC}$-measure as its inverse, the set
$$\big\{\sigma\autf_{\KP}(\FC)\;:\;\sigma\in\aut(\FC),\, \sigma(u)\in[\theta(\bar{x};\bar{b})]\big\}$$
which is the image $d_u(U)$ of $U=[\theta(\bar{x};\bar{b})]$ via a variant of
the map $d_u$ from Definition 2.2 in \cite{NEWELSKI_2014}.

\begin{definition}\label{def: almost pullback}
    Let $p\in I$, $\theta(\bar{x};\bar{y})\in\CL$ and $\bar{b}\in\FC^{\bar{y}}$.
    We define
    $$\mu_p(\theta(\bar{x};\bar{b})):=h_{\FC'}\Big( \rho\big[ [\theta(\bar{x};\bar{b})]\,\cap\,p\ast S\big] \Big).$$
\end{definition}

Note that the above defined $\mu_p$ is induced by the pullback of $h_{\FC'}$ to the Borel group $p\ast S$ via the Borel isomorphism $\rho|_{p\ast S}$. However, $p\ast S$ is not necessarily a closed subset (and so a closed subgroup)\footnote{For example, it is not the case in a real closed field expanded by an affine copy of $S^1$.} and thus $\mu_p$ is not simply the pullback of $h_{\FC'}$ but rather a way of measuring the trace of clopen/closed subsets of $S$ in $p\ast S$.
In this section, we will see that such a point of view is correct and that there is no ambiguity in computing the measure on the closure of $p\ast S$.

Let us also note that in the context of finitely satisfiable types and measures on a definable group, 
an alternative approach
was undertaken in Proposition 4.15 of \cite{CGK}, leading to Theorem 1.6/Corollary 5.18 there. This approach required additional assumption on countability of $M$, so the revised Newelski's conjecture follows and gives us that the $\tau$-topology is Hausdorff. In contrast, for invariant types, it is not even clear
what the analogue of the $\tau$-topology should be.

\begin{lemma}\label{lemma: mu_p is inv Keisler measure}
    If $p\in I$ then $\mu_p\in\mathfrak{M}^{\inv}_{\bar{n}}(\FC,\emptyset)$
    and $\supp(\mu_p)\subseteq \overline{p\ast S}$.
\end{lemma}

\begin{proof}
    We could check directly that $\mu_p\in\mathfrak{M}_{\bar{n}}^{\inv}(\FC,\emptyset)$, but 
    we give a different argument
    to explain the relation with some results of \cite{HruKruPi}.

    In the discussion after \cite[Fact 2.23]{HruKruPi}, the authors define an $\aut(\FC)$-invariant finitely additive probability measure $\mu$ on the relatively $\bar{n}$-definable subsets of $\aut(\FC)$. After applying \cite[Proposition 2.12]{HruKruPi} (and parts of its proof),
    we transfer $\mu$ over the space of types via 
    the isomorphism of the Boolean algebra of clopen subsets of $S_{\bar n}(\FC)$ and the algebra of relatively definable subsets of $\Aut(\FC)$ 
    given by
    $[\theta(\bar{x};\bar{b})]\mapsto A_{\theta,\bar{n},\bar{b}}(=\{\sigma\in\Aut(\FC)\;:\; \models \theta(\sigma(\bar n),\bar b)\})$ to obtain an $\aut(\FC)$-invariant Keisler measure $\tilde{\mu}\in\mathfrak{M}_{\bar{n}}(\FC)$.
    Then
    \begin{IEEEeqnarray*}{rCl}
        \tilde{\mu}(\theta(\bar{x};\bar{b})) &=& \mu(A_{\theta,\bar{n},\bar{b}})
        =h_{\FC}\Big(\big\{\sigma\autf_{\KP}(\FC)\;:\;\sigma(p)\in[\theta(\bar{x};\bar{b})]\big\}\Big) \\
        &=& h_{\FC}\Big(\big\{\sigma^{-1}\autf_{\KP}(\FC)\;:\;\sigma(p)\in[\theta(\bar{x};\bar{b})]\big\}\Big) \\
        &=& h_{\FC}\Big(r_{\FC}\rho\big[ [\theta(\bar{x};\bar{b})]\,\cap\,p\ast S \big] \Big) \\
        &=& h_{\FC'}\Big(\rho\big[ [\theta(\bar{x};\bar{b})]\,\cap\,p\ast S \big] \Big)=\mu_p(\theta(\bar{x};\bar{b})),
    \end{IEEEeqnarray*}
    so $\mu_p$ is the $\Aut(\FC)$-invariant measure $\tilde \mu$.

    Now, let us argue for $\supp(\mu_p)\subseteq \overline{p\ast S}$.
    If $q\in\supp(\mu_p)$ and $\theta(\bar{x};\bar{b})\in q$ then
    $$0<\mu_p(\theta(\bar{x};\bar{b}))=h_{\FC'}\Big(\rho\big[ [\theta(\bar{x};\bar{b})]\,\cap\,p\ast S \big]\Big).$$
    Hence $[\theta(\bar{x};\bar{b})]\,\cap\,p\ast S\neq\emptyset$ and so $q\in\overline{p\ast S}$.
\end{proof}

\begin{theorem}\label{thm: hereditary amenability of type flow}
    Suppose $T$ is an amenable NIP theory. Then the flow $\big(S_{\bar n}(\FC),\aut(\FC)\big)$ is hereditarily amenable.

    More precisely, if $Z\subseteq S_{\bar n}(\FC)$ is a minimal
    subflow, then $Z\subseteq I$, and for every $p\in Z$
    $\supp(\mu_p)=Z$ (where $\mu_p$ is the invariant measure given by Definition~\ref{def: almost pullback}).
\end{theorem}

\begin{proof}
    Let $Z\subseteq S_{\bar n}(\FC)$ be a minimal subflow and fix
    $p\in Z$. Let $\Sigma(\bar{x})$ be an
    $\aut(\FC)$-invariant partial type over $\FC$ such that $Z=S_{\Sigma(\bar{x})}(\FC)$.
    We first show that $p$ does not fork over $\emptyset$. Fix
    $\varphi(\bar{x};\bar{b})\in p$ and put
    $$U:=[\varphi(\bar{x};\bar{b})]\cap Z.$$
    The set $U$ is a nonempty relatively open subset of $Z$. Since
    $Z$ is minimal, the family $\{\sigma(U):\sigma\in\aut(\FC)\}$
    covers $Z$. By compactness, there are
    $\sigma_1,\ldots,\sigma_k\in\aut(\FC)$ such that
    $Z=\bigcup_{i=1}^{k}\sigma_i(U)$.
    Equivalently,
    \begin{equation*}\label{eq: generic formula in minimal subflow}
        \Sigma(\bar{x})
        \vdash
        \bigvee_{i=1}^{k}
        \varphi\big(\bar{x};\sigma_i(\bar{b})\big).
    \end{equation*}
    Since $T$ is amenable and NIP, $\emptyset$ is an extension base (by [Corollary 2.24]\cite{HruKruPi}).
    Hence forking and dividing over $\emptyset$ coincide. It follows that, since  $\Sigma(\bar{x})$ does not divide over $\emptyset$ (because it is $\aut(\FC)$-invariant), it does not fork, and neither does
    $\varphi(\bar{x};\bar{b})$.
    
    Since $\varphi$ was an arbitrary formula in $p$,
    $p$ does not fork over $\emptyset$. By
    Remark~\ref{rem: 4 properties}, $p$ is f-generic, and
    so $p\in I$. Since $p\in Z$ was arbitrary, this
    also shows that $Z\subseteq I$.

    By Lemma~\ref{lemma: mu_p is inv Keisler measure},
    $\mu_p$ is $\aut(\FC)$-invariant and
    $\supp(\mu_p)\subseteq\overline{p\ast S}$.
    Corollary~\ref{cor: orbit of f-generic} and minimality of $Z$ give
    $$\overline{p\ast S}=\overline{\aut(\FC)\cdot p}=Z.$$
    Therefore $\supp(\mu_p)\subseteq Z$. Since $Z$ is minimal, it follows that $\supp(\mu_p)=Z$.
\end{proof}

\subsection{Approximation lemmas}
Now, we need to prove several auxiliary results on approximating measures.
Let $p\in I$, $\theta(\bar{x};\bar{y})\in\CL$, $\bar{b}\in\FC^{\bar{y}}$ and let $K\subseteq I$. We set
\begin{IEEEeqnarray*}{rClCl}
    A_{\theta(\bar{x};\bar{b}),p} &:=& \{\sigma\autf_{\KP}(\FC') &:& \sigma\cdot p|_{\FC'}\in[\theta(\bar{x};\bar{b})]\}, \\
    \mathcal{F}_{\theta(\bar{x};\bar{b}),K} &:=&
    \{A_{\theta(\bar{x};\bar{b}),p}\cdot\tau\autf_{\KP}(\FC') &:& p\in K,\,\tau\in\aut(\FC')\}.
\end{IEEEeqnarray*}

\begin{lemma}\label{lemma: F finite VC dimension}
    $\CF_{\theta(\bar{x};\bar{b}),K}$ has finite VC-dimension.
\end{lemma}

\begin{proof}
    Assume that $\{\sigma_1\autf_{\KP}(\FC'),\ldots,\sigma_n\autf_{\KP}(\FC')\}$ is shattered by $\CF_{\theta(\bar{x};\bar{b}),K}$ and take $J\subseteq\{1,\ldots,n\}$.
    There exists $p_J\in K$ and $\tau_J\in\aut(\FC')$ such that
    \begin{IEEEeqnarray*}{rCl}
        j\in J &\iff & \sigma_j\autf_{\KP}(\FC')\in A_{\theta(\bar{x};\bar{b}),p_J}\cdot \tau_J\autf_{\KP}(\FC') \\
        &\iff & \sigma_j\tau_J^{-1}\autf_{\KP}(\FC')\in A_{\theta(\bar{x};\bar{b}),p_J} \\
        &\iff & \sigma_j\tau_J^{-1}\cdot p_J|_{\FC'}\ni \theta(\bar{x};\bar{b}) \\
        &\iff & \tau_J^{-1}\cdot p_J|_{\FC'}\ni \theta(\bar{x};\sigma_j^{-1}(\bar{b})).
    \end{IEEEeqnarray*}
    Now, for every $J\subseteq \{1,\ldots,n\}$, let $\bar{a}_J\models \tau_J^{-1}\cdot p_J|_{\FC'}$ (in a bigger monster $\FC''$),
    and, for every $j\in \{1,\ldots,n\}$, let $\bar{b}_j:=\sigma_j^{-1}(\bar{b})$.
    We have
    $$j\in J\quad\iff\quad\models\theta(\bar{a}_J;\bar{b}_j)$$
    and so the size of the shattered set, i.e.\ $n$, is bounded by the VC-dimension of $\theta^{\mathrm{opp}}(\bar{y};\bar{x})$.
\end{proof}

\begin{fact}[Fact 2.1 in \cite{ArtemPierre}]\label{fact: 2.1 from ArtemPierre}
    For every $k>0$ and $\epsilon>0$ there is $N<\omega$ such that:
    If $(X,\mu)$ is a probability space and $\CF\subseteq\mathcal{P}(X)$ has VC-dimension $\leqslant k$ and satisfies
    \begin{enumerate}
        \item every $F\in\CF$ is measurable,

        \item for every $n<\omega$, the function $f_n\colon X^{\times n}\to[0,1]$, given by
        $$(x_1,\ldots,x_n)\mapsto\sup\limits_{F\in\CF}|\Av(x_1,\ldots,x_n;F)-\mu(F)|,$$
        is measurable,

        \item for every $n<\omega$, the function $g_n\colon X^{\times 2n}\to[0,1]$, given by
        $$(x_1,\ldots,x_n,y_1,\ldots,y_n)\mapsto\sup\limits_{F\in\CF}|\Av(x_1,\ldots,x_n;F)-\Av(y_1,\ldots,y_n;F)|,$$
        is measurable,
    \end{enumerate}
    then there exist $x_1,\ldots,x_N\in X$ such that for every $F\in\CF$ we have
    $$\mu(F)\approx_{\epsilon}\Av(x_1,\ldots,x_N;F)\left(\coloneqq \frac1N\left|\{j\leqslant N\;:\;x_j\in F\}\right|\right).$$
\end{fact}

\begin{remark}
    \label{remark:super_monster}
    Note that given any $M\models T$ and $\kappa$, we can find an elementary extension $N\succeq M$ which is $\kappa$-saturated and such that for every $\CL_0\subseteq \CL$, $N|_{\CL_0}$ is $\kappa$-strongly homogeneous as a model of $T|_{\CL_0}$ (this follows from compactness and a standard book-keeping argument, and follows from $\kappa$-bigness as defined in \cite[Chapter 10]{Hodges}). In particular, we may assume that the monster model $\FC$ (and $\FC'$) is also a monster model for all reducts of $T$.
\end{remark}

In the proof of the following lemma, at some point, we need to show that a certain analytic set in $\gal_{\KP}(T)$ is measurable. This follows easily by a general theorem for Polish groups (saying that analytic sets are universally measurable), thus we follow the strategy from \cite[Lemma 3.21]{ArtemPierre} and pass to a countable sublanguage of $\CL$ to work with a Kim-Pillay group which is Polish.
However, in contrast to the definable group case from \cite{ArtemPierre},
we encounter some issues.
In \cite{ArtemPierre},
where the acting group $G$ (i.e.\ the definable group) remains unchanged after passing to a countable sublanguage containing the definition of $G$, 
we have that
being an f-generic in \cite{ArtemPierre} (given in terms of $G$-dividing) is preserved after passing to the sublanguage.
In our case, the acting group is the group of automorphisms, which depends on the choice of the sublanguage. 
To address this issue, we could avoid all the problems by assuming that the language $\CL$ is countable, or we could consider a special class of f-generics which remain f-generic after passing to a countable sublanguage
\footnote{Not all f-generics are like that, even not in amenable NIP. 
For example, let $(R,\leqslant)$ be a sufficiently saturated dense linear ordering, let $A\preceq R$ be a submodel of uncountable cofinality, and set $\CL_A=\{\leqslant\}\cup \{a\;:\;a\in A\}$.
Let $p\in S_1(R)$ be the type of an element ``just after $A$''. This
$p$ is $\emptyset$-invariant (in the language $\CL_A$) and so f-generic. 
On the other hand, the uncountable cofinality implies that for every countable $\CL_0\subseteq \CL$ containing $\leqslant$, in $p|_{\CL_0}$ there is a formula of the form $a\leqslant x\leqslant b$, where $b>A$ and $a\in A$ is strictly greater than the constants remaining in $\CL_0$. This formula clearly $2$-divides in $R|_{\CL_0}$. Thus, $p|_{\CL_0}$ is not an f-generic.}.
We decided to include a more general formulation of Lemma~\ref{lemma: F satisfies assumptions of Fact 2.1} 
to indicate a possible strategy
to extend our results in the future, and to assume countability of $\CL$ just after Lemma~\ref{lemma: F satisfies assumptions of Fact 2.1} to avoid a less natural formulation of several results.

\begin{lemma}\label{lemma: F satisfies assumptions of Fact 2.1}
    Let $K$ be countable and assume that for every $p(\bar{x})\in K$ there
    exists a countable $\CL_0\subseteq\CL$ containing $\theta$ such that
    $p|_{\CL_0}$ is an f-generic in $T|_{\CL_0}$ (i.e.\ $p|_{\CL_0}$ does not fork over $\emptyset$).
    
    Then $\CF_{\theta(\bar{x};\bar{b}),K}$ satisfies every point of the assumptions of Fact~\ref{fact: 2.1 from ArtemPierre} for the probability space $(\gal_{\KP}(T), h_{\FC'})$.
\end{lemma}

\begin{proof}
    We already discussed the Borelness of sets of the form $A_{\theta(\bar{x};\bar{b}),p}$ (cf.\ the proof of Lemma~\ref{lemma: rho is Borel} for $p|_{\FC'}$ which is f-generic by Lemma~\ref{lemma: f-generic extension}).
    This is enough to deduce the first point of the assumptions of Fact~\ref{fact: 2.1 from ArtemPierre}.

    Let us argue for the second point of the assumptions of Fact~\ref{fact: 2.1 from ArtemPierre}.
    Take $n<\omega$, we need to show that the map
    $f_n\colon  \gal_{\KP}(T)^{\times n}\to[0,1]$,
    \begin{IEEEeqnarray*}{rCl}
        f_n(g_1,\ldots,g_n) &=& \sup\limits_{F\in\CF_{\theta(\bar{x};\bar{b}),K}}|\Av(g_1,\ldots,g_n;F)-h_{\FC'}(F)| \\
        &=& \sup\limits_{p\in K}\;\sup\limits_{\tau\in\aut(\FC')}
        \Big|\Av\big(g_1,\ldots,g_n;A_{\theta(\bar{x};\bar{b}),p}\cdot\tau\autf_{\KP}(\FC')\big)-h_{\FC'}(A_{\theta(\bar{x};\bar{b}),p})\Big|
    \end{IEEEeqnarray*}
    is measurable. Because $K$ is countable, it is enough to show that
    $$\sup\limits_{\tau\in\aut(\FC')}
        \Big|\Av\big(g_1,\ldots,g_n;A_{\theta(\bar{x};\bar{b}),p}\cdot\tau\autf_{\KP}(\FC')\big)-h_{\FC'}(A_{\theta(\bar{x};\bar{b}),p})\Big|$$
    is a measurable map for every $p\in K$, which will follow if we can show that for a fixed $J\subseteq\{1,\ldots,n\}$ the set
    \begin{IEEEeqnarray*}{rCl}
        A_J:=\{(g_1,\ldots,g_n)\in\gal_{\KP}(T)^{\times n} &\,:\,&
    (\exists\tau\in\aut(\FC')) \\
    & & (g_j\in A_{\theta(\bar{x};\bar{b}),p}\cdot\tau\autf_{\KP}(\FC')\;\iff\; j\in J)\}
    \end{IEEEeqnarray*}
    is measurable. To do it we will pass to a Polish group via a continuous epimorphism.

    Let $\CL_0$ be a countable sublanguage of $\CL$ containing 
    $\theta(\bar{x};\bar{y})$ given in the assumptions of this lemma for $p$.    
    Set $T_0$ to be the reduct of $T$ to $\CL_0$.
    We can assume that $\FC'$ is also a monster model for $T_0$, see
    Remark~\ref{remark:super_monster}.
    The formula $\sigma\autf_{\KP}(\FC')\mapsto \sigma\autf_{
    \KP}(\FC'|_{\CL_0})$ defines a continuous homomorphism $P\colon \gal_{\KP}(T)\to\gal_{\KP}(T_0)$ (follows immediately from the fact that every type-definable set in $\FC'|_{\CL_0}$ is type-definable in $\FC'$).
    Consider a compact subgroup $\CG:=P[\gal_{\KP}(T)]\leqslant\gal_{\KP}(T_0)$, which is a Polish group since $\gal_{\KP}(T_0)$ is a Polish group.
    
    Note that the normalized Haar measure on $\CG$ is the pushforward of the Haar measure on $\gal_{\KP}(T)$.
    Indeed,
    note that the pushforward of the Haar measure is clearly invariant, and it is easy to check (using compactness) that it is inner regular, which by finiteness implies that it is also outer regular, and the conclusion follows by the uniqueness of the Haar measure.

    The preimage (via $P$) of a Haar-null set in $\CG$ is Haar-null in $\gal_{\KP}(T)$ and the preimage of a Haar-measurable set in $\CG$ is Haar-measurable in $\gal_{\KP}(T)$.
    The first part follows because a null set is contained in a null $G_\delta$ set. 
    Then, regularity implies that a measurable set is the symmetric difference of a Borel set and a null set. As preimages preserve symmetric differences, the second part follows.

    Now, we need to show that $P^{-1}[P[A_{\theta(\bar{x};\bar{b}),p}]]=A_{\theta(\bar{x};\bar{b}),p}$.
    For this, assume that $\sigma,\tau\in\aut(\FC')$ are such that 
    $P(\sigma\autf_{\KP}(\FC'))=P(\tau\autf_{\KP}(\FC'))$, i.e.
    $\sigma^{-1}\tau\in\autf_{\KP}(\FC'|_{\CL_0})$.
    If $\sigma\autf_{\KP}(\FC')\in A_{\theta(\bar{x};\bar{b}),p}$
    then $\sigma\cdot p|_{\FC'}\in[\theta(\bar{x};\bar{b})]$.
    We have that $p|_{\CL_0}$ is an f-generic in $T_0$ 
    and so $(p|_{\CL_0})|_{\FC'}=(p|_{\FC'})|_{\CL_0}$ is an f-generic in $T_0$ (by Lemma~\ref{lemma: f-generic extension}). Consequently,
    $(p|_{\CL_0})_{\FC'}$ is $\autf_{\KP}(\FC'|_{\CL_0})$-invariant
    and we obtain
    $$ \sigma\big( (p|_{\CL_0})|_{\FC'}\big)= 
    \sigma\Big(\sigma^{-1}\tau\big( (p|_{\CL_0})|_{\FC'}\big)\Big)
    =\tau\big( (p|_{\CL_0})|_{\FC'}\big).$$
    Hence,
    \begin{IEEEeqnarray*}{rCl}
        \sigma\cdot p|_{\FC'}\in[\theta(\bar{x};\bar{b})] &\iff &
        \sigma\big((p|_{\FC'})|_{\CL_0}\big) \in[\theta(\bar{x};\bar{b})] \\
        & \iff & \sigma\big( (p|_{\CL_0})|_{\FC'}\big) \in[\theta(\bar{x};\bar{b})] \\
        & \iff & \tau\big( (p|_{\CL_0})|_{\FC'}\big) \in[\theta(\bar{x};\bar{b})] \\
        &\iff & \tau\big((p|_{\FC'})|_{\CL_0}\big) \in[\theta(\bar{x};\bar{b})] \\
        &\iff & \tau\cdot p|_{\FC'}\in[\theta(\bar{x};\bar{b})]
    \end{IEEEeqnarray*}
    Therefore, if $\sigma\autf_{\KP}(\FC')\in A_{\theta(\bar{x};\bar{b}),p}$ then
    $\tau\autf_{\KP}(\FC')\in A_{\theta(\bar{x};\bar{b}),p}$ and 
    $P^{-1}[P[A_{\theta(\bar{x};\bar{b}),p}]]\subseteq A_{\theta(\bar{x};\bar{b}),p}$.
    The other inclusion follows by definition.

    Note that, by $P^{-1}[P[A_{\theta(\bar{x};\bar{b}),p}]]=A_{\theta(\bar{x};\bar{b}),p}$, we have that $g_j\in A_{\theta(\bar{x};\bar{b}),p}\cdot \tau$ if and only if $P(g_j)\in P[A_{\theta(\bar{x};\bar{b}),p}]\cdot P(\tau)$, and hence
    $$ P[A_J]=\{ (h_1,\ldots,h_n)\in\CG^{\times n}:(\exists h\in\CG)\big( h_jh^{-1}\in P[A_{\theta(\bar{x};\bar{b}),p}]\,\iff\,j\in J \big)\} $$
    Then, using again $P^{-1}[P[A_{\theta(\bar{x};\bar{b}),p}]]=A_{\theta(\bar{x};\bar{b}),p}$,
    we conclude that $P^{-1}[P[A_J]]=A_J$. Thus we are done with this part of the proof if we can show that $P[A_J]$ is Haar-measurable in $\CG$. This is our next step.

    Consider the set
    $$ A^{\CL_0}_{\theta(\bar{x};\bar{b}),p|_{\CL_0}}=\{\tau\autf_{\KP}(\FC'|_{\CL_0})\in\gal_\KP(T_0)\;:\;\tau\big(
    (p|_{\CL_0})|_{\FC'}\big)\in[\theta(\bar{x};\bar{b})]\}$$
    which is Borel since $p|_{\CL_0}$ is f-generic (even constructible, cf.\ the proof of Lemma~\ref{lemma: rho is Borel} or the discussion after Fact 2.23 in \cite{HruKruPi}). We have (using $(p|_{\CL_0})|_{\FC'}=(p|_{\FC'})|_{\CL_0}$) that $P[A_{\theta(\bar{x};\bar{b}),p}]=\CG\cap A^{\CL_0}_{\theta(\bar{x};\bar{b}),p|_{\CL_0}}$ and so $P[A_{\theta(\bar{x};\bar{b}),p}]$ is a Borel subset of $\CG$.
    It is not hard to see that $P[A_J]$ is the image under the projection $\CG^{\times(n+1)}\to\CG^{\times n}$ onto the last $n$ coordinates of the set
        \begin{IEEEeqnarray*}{rCl}
        B_J &:=& \bigcap\limits_{j\in J}\{(h,h_1,\ldots,h_n)\in\CG^{\times(n+1)}\;:\;h_jh^{-1}\in P[A_{\theta(\bar{x};\bar{b}),p}]\} \\
        &\cap & \bigcap\limits_{j\not\in J}\{(h,h_1,\ldots,h_n)\in\CG^{\times(n+1)}\;:\; h_jh^{-1}\not\in P[A_{\theta(\bar{x};\bar{b}),p}]\}.
    \end{IEEEeqnarray*}
    The set $B_J$ is a Borel subset of $\CG^{\times(n+1)}$ because
    $P[A_{\theta(\bar{x};\bar{b}),p}]$ is Borel and the group operations in $\CG$ are continuous. 
    Thus $P[A_J]$ is an analytic subset of the Polish group $\CG$ and as such, $P[A_J]$ is universally measurable (by Theorem 29(7) in \cite{Kechris95}), in particular $P[A_J]$ is Haar-measurable as expected. 

    In a similar way we can show that, for every $n<\omega$,
    the map $g_n\colon \gal_{\KP}(T)^{\times 2n}\to[0,1]$ from the third point of
    the assumptions of Fact~\ref{fact: 2.1 from ArtemPierre} is measurable. 
\end{proof}

From now until Remark~\ref{rem: Ellis stable 1}, we assume that 
\textbf{the language $\CL$ is countable}.

\begin{proposition}\label{prop: approximation 1}
    Let 
    $\epsilon>0$, $\theta(\bar{x};\bar{y})\in\CL$
    and $\bar{b}\in\FC^{\bar{y}}$,
    and let
    $K\subseteq I$ be countable.
    Then there exist $\tau_1,\ldots,\tau_n\in\aut(\FC')$ such that for every $p\in K$ and every $q\in S$
    we have
    \begin{IEEEeqnarray*}{rCl}
    \mu_{p}(\theta(\bar{x};\bar{b}))=\mu_{p\ast q}(\theta(\bar{x};\bar{b})) &\approx_{\epsilon}&
    \Av\Big(j;\; \theta\big(\bar{x};\tau_j^{-1}(\bar{b})\big)\in p\ast q|_{\FC'} \Big) \\
    &=& \frac{1}{n}\Big|\{j\leqslant n: \theta\big(\bar{x};\tau_j^{-1}(\bar{b})\big)\in p\ast q|_{\FC'}\}\Big|.
    \end{IEEEeqnarray*}
    \end{proposition}

\begin{proof}
    By Lemma~\ref{lemma: F satisfies assumptions of Fact 2.1}, there exist $\tau_1,\ldots,\tau_n\in\aut(\FC')$ such that for every $\sigma\in\aut(\FC')$ and every $p\in K$ we have
    $$h_{\FC'}\big( A_{\theta(\bar{x};\bar{b}),p}\cdot\sigma^{-1}\autf_{\KP}(\FC')\big) \approx_{\epsilon} \Av\Big(j;\;\tau_j\autf_{\KP}(\FC')\in A_{\theta(\bar{x};\bar{b}),p}\cdot\sigma^{-1}\autf_{\KP}(\FC') \Big).$$
    Note that    
    $$\tau_j\autf_{\KP}(\FC')\in A_{\theta(\bar{x};\bar{b}),p}\cdot\sigma^{-1}\autf_{\KP}(\FC')\quad\iff\quad \sigma\cdot p|_{\FC'}\ni\theta(\bar{x};\tau_j^{-1}(\bar{b})),$$
    $$h_{\FC'}\big( A_{\theta(\bar{x};\bar{b}),p}\cdot\sigma^{-1}\autf_{\KP}(\FC')\big) = h_{\FC'}( A_{\theta(\bar{x};\bar{b}),p}) $$
    and so
    \begin{equation}\label{equation1}\tag{$\dagger$}
        h_{\FC'}( A_{\theta(\bar{x};\bar{b}),p}) \approx_{\epsilon} \Av\Big(j;\; \sigma\cdot p|_{\FC'}\ni\theta\big(\bar{x};\tau_j^{-1}(\bar{b})\big) \Big),
    \end{equation}
    for every $\sigma\in\aut(\FC')$ and every $p\in K$.

    Let $q\in S$ with $q=\tp(\sigma(\bar{n})/\FC)$ for some $\sigma\in\aut(\FC')$.
    Then $p\ast q= (\sigma\cdot p|_{\FC'})|_{\FC}$, so $\sigma\cdot p|_{\FC'}=p\ast q|_{\FC'}$ (i.e.\ the unique $M$-invariant extension of $p\ast q$ over $\FC'$)
    and
    \begin{equation}\label{equation2}\tag{$\dagger\dagger$}
    \Av\Big(j;\; \sigma\cdot p|_{\FC'}\ni\theta\big(\bar{x};\tau_j^{-1}(\bar{b})\big) \Big) = \Av\big(j;\; p\ast q|_{\FC'}\ni\theta(\bar{x};\tau_j^{-1}(\bar{b})) \big).
    \end{equation}

    Now let us take care of the left-hand side, we have by Corollary~\ref{cor: orbit of f-generic}:
    \begin{IEEEeqnarray*}{rCl}
        h_{\FC'}(A_{\theta(\bar{x};\bar{b}),p}) &=& h_{\FC'}\Big(\big\{\tau\autf_{\KP}(\FC')\;:\;\theta(\bar{x};\bar{b})\in \tau\cdot p|_{\FC'}\big\}\Big) \\
        &=& h_{\FC'}\Big(\big\{\tau^{-1}\autf_{\KP}(\FC')\;:\;\theta(\bar{x};\bar{b})\in \tau\cdot p|_{\FC'}\big\}\Big) \\
        &=& h_{\FC'}\big( \big\{ \rho(r)\;:\;r\in S,\,p\ast r\in[\theta(\bar{x};\bar{b})]\big\}\Big) \\
        &=& h_{\FC'}\big( \big\{ \rho(p)^{-1}\rho(p\ast r)\;:\;r\in S,\,p\ast r\in[\theta(\bar{x};\bar{b})]\big\}\Big) \\
        &=& h_{\FC'}\Big( \rho(p)^{-1}\rho\big[ [\theta(\bar{x};\bar{b})]\,\cap\, p\ast S \big] \Big) \\
         &=& h_{\FC'}\Big( \rho\big[ [\theta(\bar{x};\bar{b})]\,\cap\, p\ast S \big] \Big) = \mu_{p}(\theta(\bar{x};\bar{b})),
    \end{IEEEeqnarray*}
    and since $p\ast q\ast S=p\ast S$  (because $p\ast S$ is a minimal right ideal, by Theorem~\ref{thm: fGen and injectivity on Ellis groups})
    we have
    \begin{equation}\label{equation3}\tag{$\dagger\dagger\dagger$}
         h_{\FC'}(A_{\theta(\bar{x};\bar{b}),p})= h_{\FC'}\Big( \rho\big[ [\theta(\bar{x};\bar{b})]\,\cap\, p\ast q\ast S \big] \Big)= \mu_{p\ast q}(\theta(\bar{x};\bar{b})).
    \end{equation}
    Putting \eqref{equation1}, \eqref{equation2}, \eqref{equation3} together, we obtain
        $$\mu_{p}(\theta(\bar{x};\bar{b})) = \mu_{p\ast q}(\theta(\bar{x};\bar{b}))\approx_{\epsilon}
    \Av\Big(j;\; \theta\big(\bar{x};\tau_j^{-1}(\bar{b})\big)\in p\ast q|_{\FC'} \Big),$$
    as expected.
\end{proof}

\begin{proposition}\label{prop: uniqueness on closure}
    If $p\in I$ and $q\in\overline{p\ast S}$ then $\mu_p=\mu_q$.
\end{proposition}

\begin{proof}
    Let $\theta(\bar{x};\bar{y})\in\CL$, $\bar{b}\in\FC^{\bar{y}}$ and $\epsilon>0$.
    By Proposition~\ref{prop: approximation 1} applied to the set $K=\{p,q\}$,
    we obtain
    $\tau_1,\ldots,\tau_n\in\aut(\FC')$ such that for every $r\in S$,
    \begin{IEEEeqnarray*}{rCl}
        \mu_{p}(\theta(\bar{x};\bar{b})) &=& \mu_{p\ast r}(\theta(\bar{x};\bar{b}))\approx_{\epsilon}
    \Av\Big(j;\; \theta\big(\bar{x};\tau_j^{-1}(\bar{b})\big)\in p\ast r|_{\FC'} \Big), \\
    \mu_{q}(\theta(\bar{x};\bar{b})) &=& \mu_{q\ast r}(\theta(\bar{x};\bar{b}))\approx_{\epsilon}
    \Av\Big(j;\; \theta\big(\bar{x};\tau_j^{-1}(\bar{b})\big)\in q\ast r|_{\FC'} \Big).
    \end{IEEEeqnarray*}
Let $r\in S$ extend $\tp(\id_{\FC}(\bar{m})/M)$. Then $q\ast r=\id_{\FC}\cdot q=q$ (by Lemma~\ref{lemma: cdot vs ast}), and so also
$$\mu_q(\theta(\bar{x};\bar{b}))\approx_{\epsilon}
\Av\Big(j;\; \theta\big((\bar{x};\tau_j^{-1}(\bar{b})\big)\in q|_{\FC'} \Big).$$

Let $\bar{y}_0$ be 
the finite subtuple consisting of all coordinates of $\bar{y}$ that occur in $\theta(\bar{x};\bar{y})$,
i.e.\ $\theta(\bar{x};\bar{y})=\theta(\bar{x};\bar{y}_0)$.
Consider $\bar{b}_0:=\bar{b}|_{\bar{y}_0}$, and for each $j\leqslant n$ take
$\bar{d}_j\in \FC^{\bar{y}_0}$ such that $\bar{d}_j\equiv_M\tau^{-1}_j(\bar{b}_0)$.
Because $q\in \overline{p\ast S}$, there exist $t\in S$ such that for every $j\leqslant n$ we have
$$\theta(\bar{x};\bar{d}_j)\in q\quad\iff\quad \theta(\bar{x};\bar{d}_j)\in p\ast t.$$
Hence, for every $j\leqslant n$, it is
$$\theta(\bar{x};\tau_j^{-1}(\bar{b}))\in q|_{\FC'}\quad\iff\quad \theta(\bar{x};\tau_j^{-1}(\bar{b}))\in p\ast t|_{\FC'}.$$
It follows that
\begin{IEEEeqnarray*}{rCl}
    \mu_q(\theta(\bar{x};\bar{b})) &\approx_{\epsilon}& \Av\Big(j;\; \theta\big(\bar{x};\tau_j^{-1}(\bar{b})\big)\in q|_{\FC'} \Big) \\
    &=& \Av\Big(j;\; \theta\big((\bar{x};\tau_j^{-1}(\bar{b})\big)\in p\ast t|_{\FC'} \Big) \\
    &\approx_{\epsilon}& \mu_{p\ast t}(\theta(\bar{x};\bar{b})) = \mu_p(\theta(\bar{x};\bar{b})).
\end{IEEEeqnarray*}
Since $\epsilon>0$ was arbitrary, we conclude that $\mu_q(\theta(\bar{x};\bar{b}))=\mu_p(\theta(\bar{x};\bar{b}))$ for every $\theta(\bar{x};\bar{y})\in\CL$ and $\bar{b}\in\FC^{\bar{y}}$.
\end{proof}

\subsection{Ergodic decomposition}

\begin{cor}\label{cor: minimal support of mu_p}
    For every $p\in I$, the support $\supp(\mu_p)$ is the unique minimal
    subflow contained in $\overline{p\ast S}$. In particular,
    $\supp(\mu_p)$ is a minimal subflow of $(I,\aut(\FC))$.
\end{cor}

\begin{proof}
    Fix $p\in I$ and put $Y_p:=\overline{p\ast S}=\overline{\aut(\FC)\cdot p}$. $Y_p$ is a subflow of
    $(I,\aut(\FC))$.

    Let $Z\subseteq Y_p$ be a minimal subflow and choose $q\in Z$. Since
    $q\in I$, Corollary~\ref{cor: orbit of f-generic} yields $Z=\overline{\aut(\FC)\cdot q}
        =\overline{q\ast S}$.
    As $q\in Y_p=\overline{p\ast S}$,
    Proposition~\ref{prop: uniqueness on closure} gives
    $\mu_q=\mu_p$. Hence, by
    Lemma~\ref{lemma: mu_p is inv Keisler measure},
    $$\supp(\mu_p)=\supp(\mu_q)\subseteq Z.$$
    By minimality of $Z$, we obtain $\supp(\mu_p)=Z$.
\end{proof}

\begin{lemma}\label{lemma: approximation 2}
    Fix $\mu\in\mathfrak{M}^{\inv}_{\bar{n}}(\FC,\emptyset)$, $\epsilon>0$ and $\theta(\bar{x};\bar{y})\in\CL$.
    Then there exist $p_1,\ldots,p_n\in\supp(\mu)\subseteq I$ such that for every $\bar{b}\in\FC^{\bar{y}}$ we have
    $$\mu(\theta(\bar{x};\bar{b}))\approx_{\epsilon}\frac{1}{n}\sum\limits_{i\leqslant n}\mu_{p_i}(\theta(\bar{x};\bar{b})).$$
\end{lemma}

\begin{proof}
    By classical results on NIP (e.g.\ Proposition 7.11 in \cite{Guide_NIP}) and Proposition~\ref{prop: support of f-generics},
    there exist $p_1,\ldots,p_n\in\supp(\mu)\subseteq I$ such that for every $\bar{b}\in\FC^{\bar{y}}$ we have
    $$\mu(\theta(\bar{x};\bar{b}))\approx_{\frac{\epsilon}{2}}\Av\Big(p_1,\ldots,p_n;\big[\theta(\bar{x};\bar{b})\big]\Big).$$
    We will show that these $p_1,\ldots,p_n$ work. 
    
    After applying Proposition~\ref{prop: approximation 1} with $K=\{p_1,\ldots,p_n\}$, for every $\bar{b}\in\FC^{\bar{y}}$ there exist $\tau_1,\ldots,\tau_m\in\aut(\FC')$ such that for every $i\leqslant n$ we have
    $$\mu_{p_i}(\theta(\bar{x};\bar{b}))
    \approx_{\frac{\epsilon}{2}}\Av\Big( \theta\big(\bar{x};\tau_j^{-1}(\bar{b})\big)\in p_i
    |_{\FC'} \Big).$$
    
    Fix $\bar{b}\in\FC^{\bar{y}}$ and take $\tau_1,\ldots,\tau_m\in\aut(\FC')$ as above.
    Let $\bar{y}_0$ be a finite subtuple of $\bar{y}$ containing all variables among $\bar{y}$ occurring in $\theta(\bar{x};\bar{y})$, and let $\bar{b}_0:=\bar{b}|_{\bar{y}_0}$. For every $j\leqslant m$ take $\bar{d}_j\in\FC^{\bar{y}_0}$ such that $\bar{d}_j\equiv_M\tau_j^{-1}(\bar{b}_0)$.
    Then
    \begin{IEEEeqnarray*}{rCl}
    \frac{1}{n}\sum\limits_{i\leqslant n}\mu_{p_i}(\theta(\bar{x};\bar{b}))
        &\approx_{\frac{\epsilon}{2}}&
        \frac{1}{n}\sum\limits_{i\leqslant n}\frac{1}{m}\sum\limits_{j\leqslant m}
        \delta_{p_i|_{\FC'}}\big( \theta(\bar{x};\tau_j^{-1}(\bar{b}))\big) \\
        &=& \frac{1}{m}\sum\limits_{j\leqslant m} \frac{1}{n}\sum\limits_{i\leqslant n}
        \delta_{p_i|_{\FC'}}\big( \theta(\bar{x};\tau_j^{-1}(\bar{b}))\big) \\
        &=& \frac{1}{m}\sum\limits_{j\leqslant m} \frac{1}{n}\sum\limits_{i\leqslant n}
        \delta_{p_i}\big( \theta(\bar{x};\bar{d}_j)\big) \\
        &=&
        \frac{1}{m}\sum\limits_{j\leqslant m} \Av\Big(p_1,\ldots,p_n;\big[\theta(\bar{x};\bar{d}_j)\big]\Big) \\
        &\approx_{\frac{\epsilon}{2}}& \frac{1}{m}\sum\limits_{j\leqslant m}\mu(\theta(\bar{x};\bar{d}_j))=
        \frac{1}{m}\sum\limits_{j\leqslant m}\mu|_{\FC'}\big(\theta(\bar{x};\tau_j^{-1}(\bar{b}))\big) \\
        &=& \frac{1}{m}\sum\limits_{j\leqslant m}(\tau_j\cdot\mu|_{\FC'})\big(\theta(\bar{x};\bar{b})\big)
        =\frac{1}{m}\sum\limits_{j\leqslant m}\mu\big(\theta(\bar{x};\bar{b})\big) \\
        &=& \mu(\theta(\bar{x};\bar{b})),
    \end{IEEEeqnarray*}
    where $\mu|_{\FC'}$ is the unique $\emptyset$-invariant extension of $\mu$ over $\FC'$.
\end{proof}

\begin{cor}\label{cor: uniques over closure of Ellis group}
    Let $\mu\in\mathfrak{M}^{\inv}_{\bar{n}}(\FC,\emptyset)$
    and $p\in I$.
    If $\supp(\mu)\subseteq \overline{p\ast S}$ then $\mu=\mu_p$.
\end{cor}

\begin{proof}
    Let $\epsilon>0$, $\theta(\bar{x};\bar{y})\in\CL$, $\bar{b}\in\FC^{\bar{y}}$.
    By Lemma~\ref{lemma: approximation 2}, there exist $p_1,\ldots,p_n\in\supp(\mu)\subseteq \overline{p\ast S}$ such that
    $$\mu(\theta(\bar{x};\bar{b}))\approx_{\epsilon}\frac{1}{n}\sum\limits_{i\leqslant n}\mu_{p_i}(\theta(\bar{x};\bar{b})).$$
    Because $p_i\in \overline{p\ast S}$, Proposition~\ref{prop: uniqueness on closure} implies that $\mu_{p_i}=\mu_p$ for every $i\leqslant n$.
    Now, the conclusion follows easily.
\end{proof}

\begin{cor}\label{cor: minimal subflows uniquely ergodic}
Every minimal subflow of $\big(S_{\bar n}(\FC),\aut(\FC)\big)$
    is uniquely ergodic.
    More precisely, if $Z\subseteq S_{\bar n}(\FC)$ is a minimal
    subflow and $p\in Z$, then $p\in I$ and
    $\supp(\mu_p)=Z=\overline{p\ast S}$.
    Moreover, $\mu_p$ is the unique $\aut(\FC)$-invariant probability
    measure concentrated on $Z$.
\end{cor}

\begin{proof}
    By Theorem~\ref{thm: hereditary amenability of type flow}, $Z\subseteq I$.
    Fix $p\in Z$. By Corollary~\ref{cor: orbit of f-generic},
    $\aut(\FC)\cdot p=p\ast S$.
    Since $Z$ is minimal and contains $p$, we have
    $$Z=\overline{\aut(\FC)\cdot p}=\overline{p\ast S}.$$
    Theorem~\ref{thm: hereditary amenability of type flow} also gives $\supp(\mu_p)=Z$.

    Now let $\nu\in\mathfrak{M}^{\inv}_{\bar n}(\FC,\emptyset)$ be concentrated on $Z$. Then $\supp(\nu)\subseteq Z=\overline{p\ast S}$,
    so
    by Corollary~\ref{cor: uniques over closure of Ellis group}
    yields $\nu=\mu_p$, so $Z$ is uniquely ergodic.
\end{proof}

\begin{lemma}\label{lemma: continuity of mu_p}
    The map
    $$I\ni p\mapsto \mu_p\in\mathfrak{M}^{\inv}_{\bar{n}}(\FC,\emptyset)$$
    is continuous.
\end{lemma}

\begin{proof}
    Let $\theta(\bar{x};\bar{y})\in\CL$, $\bar{b}\in\FC^{\bar{y}}$, $r,s\in\Rr$,
    and consider
    $$Y:=\{p\in I\;:\;r\leqslant\mu_p(\theta(\bar{x};\bar{b}))\leqslant s\}.$$
    We want to show that $Y$ is closed in $I$, i.e.\ $Y=\overline{Y}$.

    Let $q\in\overline{Y}$, so there is a net $(p_j)_{j\in J}\subseteq Y$ converging to $q$.
    Let $M_0\preceq\FC$ be countable.
    Since $q,p_j$ are f-generics, they are $M_0$-invariant.
    In particular,
    $q|_{\theta},p_j|_{\theta}\in \mathrm{Inv}_{\theta}(M_0)$ and
    $$q|_{\theta}\in\overline{\{p_j|_{\theta}\;:\;j\in J\}}.$$
    As $\mathrm{Inv}_{\theta}(M_0)$ is a Rosenthal compactum, there is a sequence
    $(p_{j_n})_{n<\omega}$ such that $(p_{j_n}|_{\theta})_{n<\omega}\to q|_{\theta}$ in $\mathrm{Inv}_{\theta}(M_0)$ (cf.\ Proposition 2.15 in \cite{Simon_2015}; note that here we use the hypothesis that the language is countable).

    Let $\bar{y}_0$ be a finite subtuple of $\bar{y}$ containing all variables among $\bar{y}$ which occur in $\theta(\bar{x};\bar{y})$, let $\bar{b}_0:=\bar{b}|_{\bar{y}_0}$.
    Consider $\epsilon>0$ and apply Proposition~\ref{prop: approximation 1} to the set $K=\{q\}\cup\{p_{j_n}\;:\;n<\omega\}$ and $\frac{\epsilon}{2}$ to find $\tau_1,\ldots,\tau_m\in\aut(\FC')$
    and
    $\bar{d}_k\in\FC^{\bar{y}_0}$
    such that for every $k\leqslant m$ we have
    $\bar{d}_k\equiv_M\tau_k^{-1}(\bar{b}_0)$, and
    \begin{IEEEeqnarray*}{rClCl}
            \mu_{q}(\theta(\bar{x};\bar{b})) &\approx_{\frac{\epsilon}{2}}&
    \Av\Big(k;\; \theta\big(\bar{x};\bar{d}_k\big)\in q\Big) &=&
    \Av\Big(k;\; \theta\big(\bar{x};\bar{d}_k\big)\in q|_{\theta}\Big), \\
    \mu_{p_{j_n}}(\theta(\bar{x};\bar{b})) &\approx_{\frac{\epsilon}{2}}&
    \Av\Big(k;\; \theta\big(\bar{x};\bar{d}_k\big)\in p_{j_n}\Big) &=&
    \Av\Big(k;\; \theta\big(\bar{x};\bar{d}_k\big)\in p_{j_n}|_{\theta}\Big),
    \end{IEEEeqnarray*}
    for every $n<\omega$.
    There is $n_0<\omega$ such that for every $n\geqslant n_0$
    $$\Av\Big(k;\; \theta\big(\bar{x};\bar{d}_k\big)\in p_{j_n}|_{\theta}\Big) = \Av\Big(k;\; \theta\big(\bar{x};\bar{d}_k\big)\in q|_{\theta}\Big).$$
    Therefore, for every $n\geqslant n_0$ we have
    $$\mu_{p_{j_n}}(\theta(\bar{x};\bar{b}))\approx_{\epsilon}\mu_q(\theta(\bar{x};\bar{b})).$$
    Because each $p_{j_n}$ belongs to $Y$, we have that $r\leqslant\mu_{p_{j_n}}(\theta(\bar{x};\bar{b}))\leqslant s$, and thus (since $\epsilon$ was arbitrary) $\mu_q(\theta(\bar{x};\bar{b}))\in \bigcap\limits_{\epsilon>0}(r-\epsilon,s+\epsilon)=[r,s]$. Hence $q\in Y$.
\end{proof}

\begin{cor}\label{cor: mu_p flows}
    \begin{enumerate}
        \item $\{\mu_p\;:\;p\in I\}$ is a closed subset of $\mathfrak{M}^{\inv}_{\bar{n}}(\FC,\emptyset)$.

        \item For every $\mu\in\mathfrak{M}^{\inv}_{\bar{n}}(\FC,\emptyset)$, the set $\{p\in I\;:\;\mu_p=\mu\}$ is a subflow in $(I,\aut(\FC))$.
    \end{enumerate}
\end{cor}

\begin{proof}
    The set considered in the first point 
    is the image of a compact set under a continuous map
    (Lemma~\ref{lemma: continuity of mu_p}), thus it is closed.

    Now, we argue for the second point. Let $\sigma\in\aut(\FC)$ and take for $q\in S$ a coheir of $\tp(\sigma(\bar{m})/M)$. Then, by Lemma~\ref{lemma: cdot vs ast}, we have
    $\sigma\cdot p=p\ast q\in p\ast I$. This implies $\mu_{\sigma\cdot p}=\mu_{p\ast q}=\mu_p$.
    Moreover, $\{p\in I\;:\;\mu_p=\mu\}$ is closed as the preimage of $\{\mu\}$ via a continuous map.
\end{proof}

\begin{theorem}\label{thm: ergodic description}
    Let $T$ be a countable amenable NIP theory. Then the ergodic measures in $\mathfrak{M}^{\inv}_{\bar{n}}(\FC,\emptyset)$ are exactly the measures $\mu_p$ (cf.\ Definition~\ref{def: almost pullback}) for $p\in \fGen$, i.e.:
    $$ \{\mu_p\;:\;p\in \fGen\} = \{\mu\in\mathfrak{M}^{\inv}_{\bar{n}}(\FC,\emptyset)\;:\;\mu\text{ is ergodic in }\mathfrak{M}^{\inv}_{\bar{n}}(\FC,\emptyset)\}$$
\end{theorem}

\begin{proof}
    ($\subseteq$): Let $p\in I=\fGen$. We claim that $\mu_p\in\mathfrak{M}^{\inv}_{\bar{n}}(\FC,\emptyset)$ is ergodic. Suppose, towards contradiction, 
    that $\mu_p=t\mu_1+(1-t)\mu_2$ for some $\mu_1,\mu_2\in\mathfrak{M}^{\inv}_{\bar{n}}(\FC,\emptyset)$ and $0<t<1$. Then $\supp(\mu_1),\supp(\mu_2)\subseteq\supp(\mu_p)\subseteq \overline{p\ast S}$. Then $\mu_1=\mu_p=\mu_2$ by Corollary~\ref{cor: uniques over closure of Ellis group}.

    ($\supseteq$): 
    Let $Y:=\{\mu_p\;:\;p\in I\}$. 
    By Lemma~\ref{lemma: approximation 2} and the standard Shelah's trick,
    $\mathfrak{M}^{\inv}_{\bar{n}}(\FC,\emptyset)=\overline{\conv(Y)}$.
    Then a classical fact (e.g.\ Fact 4.1 in \cite{ArtemPierre}) implies that all the extreme points
    of $\mathfrak{M}^{\inv}_{\bar{n}}(\FC,\emptyset)$
    (i.e.\ the ergodic measures) are contained
    in the closure of $Y$, which is equal to $Y$
    by Corollary~\ref{cor: mu_p flows}(1).
\end{proof}

\subsection{Stable case}
For the next two remarks, we do not assume countability of the language.

\begin{remark}\label{rem: Ellis stable 1}
    Fix $p\in I$. Then:
    \begin{enumerate}
        \item If $T$ is stable, then $p\ast-\colon S^{\inv}_{\bar{n}}(\FC,M)\to S^{\inv}_{\bar{n}}(\FC,M)$ is continuous and $p\ast I$ is closed.

        \item
        If $p\ast I$ is closed in $S_{\bar n}(\FC)$, then $(p\ast I,\ast)$ is a Hausdorff compact group isomorphic to $\gal_\KP(T)$.

        \item
        If $p\ast I$ is closed in $S_{\bar n}(\FC)$, then $\mu_p$ restricted to $p\ast I$ is the unique Haar measure on $p\ast I$, in particular $\supp(\mu_p)=p\ast I$.
    \end{enumerate}
\end{remark}

\begin{proof}
    For the proof of the first point, note that $\ast$-product is right-continuous in the stable case, which follows easily from the definability of types invariant over models. Then $p\ast I$ is compact as the image of the compact set $I$ under the continuous map $p\ast-$.

    Now, we proceed to the proof of the second point.
    If $p*I$ is compact, then $\rho\colon p*I\to \gal_\KP(T)$ is a homeomorphism (as a continuous bijection between compact spaces), so $p*I$ is a compact topological group (with the induced topology) isomorphic to $\gal_\KP(T)$.
\end{proof}

\begin{remark}
    Arguing as in \cite[Proposition 7.23.(1)]{GHK}, we can show that if $T$ is stable, then $\fGen$ has a unique idempotent and so $(\fGen,\ast)$ is a profinite topological group isomorphic to $\gal_{\mathrm{Sh}}(T)=\gal_\KP(T)$.
\end{remark}

\appendix

\section{Semigroups}
\begin{remark}\label{rem: S in I}
    If $S\leqslant S'$ is a simple semigroup and $I$ is a two-sided ideal in $S'$ intersecting $S$, then $I$ contains $S$: if $s\in S\cap I$, then $S=SsS\subseteq S'sS'\subseteq I$.
\end{remark}

\begin{remark}
    If $S\leqslant S'$ is a subsemigroup and $I$ is the minimal ideal of $S'$, $S$ intersects $I$, and $S$ has a minimal ideal. Then $S\cap I$ is the minimal ideal of $S$. (Indeed, by minimality, $S\cap I$ contains the minimal ideal of $S$, which is left simple.)
\end{remark}

\begin{proposition}\label{prop: unique left ideal in subsemigroup}
    Suppose $S'$ is a compact left topological semigroup, $S\leqslant S'$ is a closed subsemigroup. Suppose furthermore that there is exactly one minimal left ideal $I$ of $S'$ intersecting $S$. Then $I\cap S$ is the unique minimal left ideal of $S$.
\end{proposition}
\begin{proof}
    Notice that $I\cap S$ is left simple: it is a closed subsemigroup of $I$, so it is compact, and as such, it contains a minimal left ideal $I''$ of $I\cap S$. We claim that $I''=I\cap S$. Indeed, let $u\in I''$ be idempotent, and let $s\in I\cap S$ be arbitrary. Then, since $u,s\in I$, which is a minimal left ideal in $S'$, we have that $su=s\in (I\cap S)I''=I''$.

    On the other hand, $I\cap S$ is the intersection with $S$ of the minimal two-sided ideal of $S'$ (since $I$ is the unique minimal left ideal of $S'$ intersecting $S$), and so it is a two-sided ideal of $S$. Since it is left simple, it follows that it is a minimal left ideal in itself, and hence also in $S$.
\end{proof}

\begin{proposition}\label{prop: ideals vs semigroups}
    Suppose $S'$ is a compact left topological semigroup, $S\leqslant S'$ closed subsemigroup which intersects the minimal ideal of $S'$. Then Ellis groups of $S$ are subgroups of Ellis groups of $S'$.
\end{proposition}
\begin{proof}
    Since $S$ is closed, it is compact, so it has a minimal ideal. It follows that the minimal ideal in $S$ is contained in the minimal ideal of $S'$. Since the Ellis groups are simply the maximal subgroups of the minimal ideal, the conclusion follows.
\end{proof}

\printbibliography
\end{document}